\documentclass{article}

\usepackage{amsfonts}
\usepackage[T1]{fontenc}
\usepackage{lmodern}
\usepackage{amsmath,amsthm,amssymb,bm}
\usepackage{anysize}
\usepackage{mathtools}
\usepackage{enumitem}

\usepackage[hidelinks]{hyperref}
\numberwithin{equation}{section}
\usepackage{cleveref,nameref}
\usepackage{xcolor}
\newtheorem{theorem}{Theorem}[section]
\newtheorem{lemma}[theorem]{Lemma}
\newtheorem{proposition}[theorem]{Proposition}
\newtheorem{corollary}[theorem]{Corollary}

\theoremstyle{definition}
\newtheorem{definition}[theorem]{Definition}
\newtheorem{remark}[theorem]{Remark}
\newtheorem{remarks}[theorem]{Remarks}
\newtheorem{example}[theorem]{Example}

\newcommand{\Ls}{\mathrm{L}_s}
\newcommand{\Gs}{\mathrm{G}_s}
\newcommand{\GsV}{\mathrm{G}_{s,V}}
\newcommand{\dist}{{\rm dist}\,}
\DeclareMathOperator{\Dom}{Dom}

\DeclareMathOperator{\sign}{sign}
\DeclareMathOperator{\defeq}{\dot{=}}

\newcommand{\ee}{\varepsilon}
\newcommand{\Xs}{X^s_\Omega}

\begin{document}

\title{The fractional Schr\"odinger equation with general nonnegative potentials. The weighted space approach}

\author{J.I. D\'{\i}az\thanks{Instituto de Matem\'atica Interdisciplinar,
Universidad Complutense de Madrid,
Plaza de Ciencias, 3, 28040 Madrid (Spain). E-mail address:  ji${}_{-}$diaz@mat.ucm.es}
\and D. G\'omez-Castro\thanks{Instituto de Matem\'atica Interdisciplinar,
Universidad Complutense de Madrid,
Plaza de Ciencias, 3, 28040 Madrid (Spain). E-mail address: dgcastro@ucm.es}
\and J.L. V\'azquez\thanks{Dpto. Matem\'aticas, Univ. Aut\'onoma de Madrid, 28049  Madrid (Spain). E-mail address:  juanluis.vazquez@uam.es}
}

\date{}
\maketitle
\vspace{-.5cm}
\begin{center}
	\emph{Dedicated to Professor Carlo Sbordone on the occasion of his 70th birthday}
\end{center}
\

\begin{abstract} We study the Dirichlet problem for the stationary Schr\"odinger fractional Laplacian equation \ $(-\Delta)^s u + V u = f$ posed in bounded domain $  \Omega  \subset \mathbb R^n$ with zero outside conditions. We consider general nonnegative potentials $V\in L^1_{loc}(\Omega)$ and prove well-posedness of very weak solutions when the data are chosen in an optimal class of weighted integrable functions $f$. Important properties of the solutions, such as its boundary behaviour, are derived.  The case of super singular potentials that blow up near the boundary is given special consideration since it leads to so-called flat solutions. We comment on related literature.
\end{abstract}

\tableofcontents

\

\newpage

\section{Introduction}

Over the last decades there has been a strong research effort devoted to extend the theory of elliptic and parabolic equations to models in which the Laplacian operator or its elliptic equivalents are replaced by different types of nonlocal integro-differential operators, most notably those called fractional Laplacian operators, given by the formula
\begin{equation}\label{frac.lap}
	(-\Delta)^s u (x) = c_{n,s} P.V. \int_{\mathbb R^n} \frac{u(x) - u(y)}{|x-y|^{n + 2s}}dy\,,
\end{equation}
with parameter $s\in (0,1)$ and a precise constant $ c_{n,s}>0 $ that we do not need to make explicit. In this formula the domain of definition is assumed to be $\mathbb R^n$.   The operator can also be defined via the Fourier transform on $\mathbb R^n$, see the classical references \cite{landkof1972foundations,stein2016singular}. With an appropriate value of the constant $c_{n,s}$, the limit $s \to 1$ produces the classical Laplace operator $-\Delta$, while the limit $s \to 0$ is the identity operator. An equivalent definition of this fractional Laplacian uses the so-called extension method, that was well-known for $s=\frac 1 2$ and has been extended to all $s\in (0,1)$ by Caffarelli and Silvestre \cite{caffarelli+silvestre2007extension}. In view of its interest in different applications, many authors have taken part in such an effort from different points of view:  probability, potential theory, and PDEs. We are interested in the PDE point of view and its connection with questions of Functional Analysis.

In this paper we study the fractional elliptic equation of Schr\"odinger type
\begin{equation} \label{eq:FDE} \tag{P}
	\begin{dcases}
		(-\Delta)^s u + V u = f & \Omega,  \\
		u = 0 & \mathbb R^n \setminus \Omega.
	\end{dcases}
\end{equation}
Here, $\Omega$ is a bounded subdomain of the space $\mathbb R^n$, $n\ge 2$, with $C^2$ boundary, and the fractional Laplacian operator $(-\Delta)^s$ is the so-called restricted or natural version given by the formula \eqref{frac.lap}, where now $x\in \Omega$ while $y$ extends to the whole space. The potential $V \ge 0$ is a measurable function satisfying mild integrability assumptions. The aim is to treat general classes of data $f$ and potentials $V$, in particular very singular potentials that blow up near the boundary and appear in important applications.

We will start from the Dirichlet problem for the fractional Laplacian equation
\begin{equation} \label{eq:FDE Laplace} \tag{P$^{0}$}
\begin{dcases}
(-\Delta)^s u = f & \Omega,  \\
u = 0 & \mathbb R^n \setminus \Omega\,,
\end{dcases}
\end{equation}
i.e., the case of zero potential. This problem and variants thereof have been well studied  by many authors, we refer to the excellent survey \cite{Ros-Oton2016}, which contains many basic references, see also \cite{BFV2018,Caffarelli2017, chen+veron2014, Ros-Oton2014}. Here, it will serve to introduce concepts and results and pose the theory for optimal classes of data $f$. We abandon the usual weak solutions of the energy theory and consider locally integrable functions $f$. This leads to the theory of {\sl very weak} and {\sl dual solutions}. The optimal class of data turns out to be the class of weighted integrable functions
\begin{equation}
L^1(\Omega; \delta^s )=\{f \, \mbox{measurable in } \ \Omega: \ f \delta^s \in L^1( \Omega)\},
\end{equation}
where $\delta (x)=\dist(x,\Omega^c)$. The problem of well-posedness with weighted integrable data has been considered in the case $s=1$, which we will call classical case hereafter, by Brezis in \cite{Brezis1971,brezis+cazenave+martel1996blow+up+revisited} in the framework of very weak solutions in weighted spaces, and has been treated recently by several authors,
\cite{Diaz+Rakotoson:2009, Diaz+Rakotoson:2010, Orsina2017, Rakotoson2012}. Extending that work to the fractional case is an important issue that we address. We recall the existence and uniqueness of very weak solutions, and establish the main properties in the case where $f\in L^1(\Omega; \delta^s )$, like  comparison, accretivity, boundary behavior, Hopf principle, and optimality of data. We devote special attention to the question of clarifying the existence of traces of the very weak solutions: we find the condition on $f$ for the trace to exist in an integral sense.

We then address a key issue of this paper, i.\,e., the study of the stationary Schr\"odinger equation. We want to  solve Problem \eqref{eq:FDE} with general $V\ge 0$ and $f$. Here the concept of very weak solution plays an important role. The theory is simple in the class of bounded or integrable potentials. Besides, if the potential is moderately singular, in a sense to be specified later, Hardy's inequality and Lax-Milgram implies uniqueness of weak solutions. However, we are interested precisely in a type of potentials that diverges at the boundary, and this leads to a delicate analysis.  Theorems \ref{thm:uniqueness} and \ref{thm:existence} settle the well-posedness of the Dirichlet problem in the class of very weak solutions for general data and count among the main results of this paper. See the whole Section 4 for further results. By the way, our results also improve what was known for $s=1$.

Here is a main motivation for the interest in general potentials.
For the classical Laplacian ($s=1$),  it was first shown by Sir Nevill Francis Mott in his 1930 book  \cite{mott1930outline}, inspired in the pioneering paper by Gamow \cite{Gamow1929} on the tunneling effect, \normalcolor that for certain families of potentials the Schr\"odinger equation, which is naturally posed in $\mathbb R^n$, can be localized to a bounded domain of $\mathbb R^n$, which is given by the nature of $V$.
On the other hand, the quasi-relativistic approach to bounded states of the Schr\"{o}dinger equation, leads to the fractional
operator corresponding to $s=1/2,$ $\sqrt{(-\Delta ) + m^2}u$ (which is also known as \emph{Klein-Gordon square root operator}, see,e.g. \cite{frank+lieb2007,herbst1977spectral}, see also \cite{Fall2015,KALETA2013} and  references). In the case of massless particles we obtain $(-\Delta)^{\frac 1 2} u$, also called in this context \emph{ultra-relativistic operator}. \normalcolor Some illustrative examples of the
class of singular potentials to which we want to apply our results are the ones given, for instance, by the attractive Coulomb case for a charge distributed over $\partial \Omega$: $V(x)=C/\delta (x)$ with $C>0$, or even by more singular functions
as it is the case of the P\"{o}sch-Teller potential (see \cite{Pochl1933})
\begin{equation}
V(x)=V(\left\vert x\right\vert )=\frac{1}{2}V_{0}\left ( \frac{k(k-1)}{\sin
	^{2}\alpha \left\vert x\right\vert }+\frac{\mu (\mu -1)}{\cos ^{2}\alpha
	\left\vert x\right\vert } \right) \label{Posch-Teller}
\end{equation}%
for some $V_{0},\alpha >0,$ $k,\mu \geq 0$ , intensively studied since 1933. Notice that this potential  blows up in a sequence of spheres.
Some other singular potentials, in the class of the so called \textit{%
super-symmetric potentials} can be found, e.g. in \cite{Cooper1995}. We refer to \cite{AizenMR644024}  as a classical paper on the mathematical study of the time independent Schr\"odinger equation for the standard Laplacian. See also \cite{Ponce2016a} for a recent reference.

Actually, another key point of this paper is studying  the sense in which the solutions of \eqref{eq:FDE} satisfy the boundary condition $u=0$ on $\partial \Omega,$ see Section \ref{sec.sspot} where so-called \emph{ flat solutions} are discussed, and also its interplay with the way in which the extended function (defined in the whole space $\mathbb{R}^{n}$) satisfies (or not) the same partial differential equation. This plays an important role in many applications as for instance Quantum Mechanics, as already mentioned.
Furthermore, since in bounded domains there are several different choices of $(-\Delta)^s$ present in the literature (see, e.g., \cite{BFV2018, MR3588125}), it is relevant to study which choice represents the correct localization of a global problem (see \Cref{sec:natural limit}).

Our data $f$ belong to an optimal class of locally integrable functions. There is a simple extension of the theory to cover the case where integrable functions $f$ are replaced by  measures $\mu$. The precise space is ${\mathcal M}(\Omega, \delta^s)$ consisting of locally bounded signed Radon measures $\mu$ such that \
$
\int_\Omega \delta^s(x)d|\mu|(x)<+\infty.
$
Actually, the results of \cite{chen+veron2014} that cover the zero-potential case are written in that generality. Our existence and uniqueness theory, contained in  Theorems \ref{thm:uniqueness} and \Cref{thm:existence}, is valid in that context. We have refrained  from that generality in our presentation because using functions makes most of our calculations and consequences easier to formulate.

Regarding potentials, we have considered general nonnegative potentials $V\in L^1_{loc}(\Omega)$. This class allows for extensions in two directions: considering signed potentials, and considering locally bounded measures as potentials. Both are present in the literature, but both lead to problems that we did not want to consider here.

\medskip

{\sc Comment.} The paper surveys topics that are treated, at least in part, in the recent literature, but it is also a research paper and many results are new, specially in Sections 3, 4, 5 and 6. We have tried to mention suitable references to relevant and related known results. Since the literature on elliptic problems with fractional Laplacians is so numerous, we refer to specialized monographs for
 more complete bibliographical information and beg excuse for possible undue omissions.


\section{Preliminaries}\label{sec.prelim}

 We introduce the fractional seminorm
\begin{equation}
[v]_{H^s (\mathbb R^n) }^2 =  \int_ {\mathbb R^n } \int_ {\mathbb R^n } \frac{|f(x)-f(y)|^2}{|x-y|^{n+2s}} dy dx ,
\end{equation}
and then the fractional Hilbert spaces $H^s$ defined by
\begin{equation}
H^s ({\mathbb R^n }) = \{  v \in L^2 ({\mathbb R^n }) :  [v]_{H^s ({\mathbb R^n })} <+ \infty  \},
\end{equation}
with the norm
\begin{align}
\|v\|_{H^s ({\mathbb R^n })}^2 &= \|v \|_{L^2({\mathbb R^n })}^2 + [v]_{H^s ({\mathbb R^n })}^2.
\end{align}
We point out that
\begin{equation}
\|v\|_{H^s({\mathbb R^n })} \asymp \|v\|_{L^2({\mathbb R^n })} + \|(-\Delta)^{\frac s 2}v\|_{L^2({\mathbb R^n })}  .
\end{equation}
where the symbol $a \asymp b$ means that there are constants $c_1, c_2 > 0$ such that $c_1 a \le b \le c_2 a$.

When working in a bounded domain $\Omega$, and in order to take into account the boundary and exterior conditions, we  define the Hilbert spaces
\begin{align}
H_0^s (\Omega) &=  \overline{\mathcal C_c^\infty (\Omega)}^{\| \cdot \|_{H^s (\mathbb R^n)}}.
\end{align}
  Classical texts on Sobolev spaces to be consulted are \cite{Adams+Fournier:2003sobolev+spaces,brezis2010functional,Leoni1967Sobolev,lions+magenes1972,Tartar2007}.  For a concise introduction to $H^s (\mathbb R^n)$ we refer the reader to \cite{Bonforte2015, DiNezza2012}.

By analogy to the classical case $s =1$, a ``formula of integration by parts'' (or ``Green's formula'') holds
\begin{proposition} \label{prop:fractional integration by parts}
	Let $u, v \in C_c^\infty (\mathbb R^n)$ and $0 < s \le 1$. Then
	\begin{equation} \label{eq:integration by parts}
	\int_ {\mathbb R^n} v ( - \Delta)^s u = \int_ {\mathbb R^n} (-\Delta)^{\frac s 2} u ( -\Delta)^{\frac s 2 } v.
	\end{equation}
\end{proposition}
\begin{proof}
	Since the operator $(-\Delta)^{\frac s 2}$ is self-adjoint (see, e.g., \cite{chen+veron2014}) and $(-\Delta)^{s+t} = (-\Delta)^s (-\Delta)^t$ for $0 < s,t,s+t \le 1$ we have
	\begin{equation}
		\int_ {\mathbb R^n} v ( - \Delta)^s u = \langle v, (-\Delta)^{\frac s 2} ( - \Delta)^{\frac s 2} u \rangle = \langle (-\Delta)^{\frac s 2} v, (- \Delta)^{\frac s 2} u \rangle =  \int_ {\mathbb R^n} (-\Delta)^{\frac s 2} u ( -\Delta)^{\frac s 2 } v.
	\end{equation}
	This proves the result.
\end{proof}

More general integration by parts results can be found in \cite{abatangelo2013large}, that treats a general integration by parts formula that includes terms accounting for a non-zero value of $u$ in $\Omega^c$ and a precise limit on the boundary.

\begin{remark}
	By density, the formula is true for any $u \in H^{2s} (\Omega) \cap H_0^s(\Omega)$ and $v \in H^s_0 (\Omega)$. Particularising for $s=1$ and making $u = v$ we deduce the classic formula
	\begin{equation}
	\| \nabla v \|_{L^2(\Omega)^n}= \| (-\Delta)^{\frac 1 2} v \|_{L^2 ({\mathbb R^n })}.
	\end{equation}
\end{remark}
\begin{remark}
	Some authors prefer the following presentation:
	\begin{equation}\label{double}
			\int_ {\mathbb R^n} v ( - \Delta)^s u  = c_{n,s} P.V.  \int_ {\mathbb R^n} \int_ {\mathbb R^n} \frac{ ( v(x) - v(y) ) (u(x) - u(y) ) }{|x-y|^{n+2s} } dy dx.
	\end{equation}
\end{remark}

 This operator also has a Kato inequality. In \cite{Caffarelli2017}, is presented simply as: $	(-\Delta)^s |u| \le \sign u \, (-\Delta)^s u$ holds in the distributional sense. A precise expression can be found in \Cref{lem:Kato}.
In  \Cref{section:Kato proof} we provide for the reader's benefit a simple proof which is useful for our  presentation.

\section{Dirichlet problem  without potentials and general data}\label{sec.nopot}

\subsection{Weak solutions. A survey on existence, uniqueness and properties}

 A weak solution of the Dirichlet  problem \eqref{eq:FDE Laplace} (with zero potential $V=0$) can be obtained by an energy minimization method using the appropriate fractional Sobolev spaces, as introduced in the previous section.
 Applying \eqref{eq:integration by parts} we introduce the concept of weak solution as:

\begin{definition}  \label{defn:ws no potential}
	$f \in L^2 (\Omega)$. A weak solution of \eqref{eq:FDE Laplace} is a function $u \in H^s_0 $ such that
	\begin{equation} \tag{P$^{0}_{\textrm{w}}$}	\int_ {\mathbb R^n } (-\Delta)^{\frac s 2} u ( -\Delta)^{\frac s 2 } \varphi \,dx= \int_{ \Omega } f \varphi\,dx, \qquad \varphi \in H_0^s (\Omega).
	\end{equation}
\end{definition}

Existence and uniqueness of weak solutions is easy for $f \in L^2 (\Omega)$, by the Lax-Milgram theorem. This is a basic result on which the extended theory is based. Actually, $f$ can be taken in the dual space $(H_0^s(\Omega))'$.

Another option is pursued in  \cite{Caffarelli+Silvestre2009,Caffarelli+Silvestre2011Archive, Caffarelli+Silvestre2011Annals} where the authors proved existence and regularity of viscosity solutions. Both classes of solutions coincide in the common class of data. We will not deal with viscosity solutions in this paper. See also \cite{Cabre2014}.

In \cite{Chen1998,Kulczycki1997}, the authors prove that the solution operator is given by an integral representation in terms of a Green kernel
\begin{equation} \label{eq:Green operator}
[(-\Delta)^{s}]^{-1} f = \int_{\Omega} \mathbb G_s (x,y) f(y) dy
\end{equation}
where
\begin{equation} \label{eq:Green function estimates}
\mathbb G_s(x,y) \asymp \frac{1}{|x-y|^{n-2s}} \left( \frac{\delta(x)}{|x-y|} \wedge 1 \right)^s \left( \frac{\delta(y)}{|x-y|} \wedge 1 \right)^s.
\end{equation}
Using these bounds, many estimates of integrability and regularity can be given by suitably applying H\"older's inequality.

The following regularity results are proved by Ros-Oton and Serra in \cite{Ros-Oton2014} and will be essential in what follows.

\begin{proposition} \label{prop:Ros-Oton} Let $\Omega$ be a bounded $C^{1,1}$, $f\in L^\infty(\Omega),$ and let $u$ be a weak solution of \eqref{eq:FDE Laplace}. Then, the following holds:
\begin{enumerate}
\item We have  $u\in C^s(\mathbb R^n)$ and
\begin{equation}
\|u\|_{C^s (\mathbb R^n)}\le C \|f\|_{L^\infty (\Omega)}\,,
\end{equation}
where $C$ is a constant depending only on $\Omega$ and $s$.

\item Moreover, if $\delta (x)=\dist(x,\Omega^c)$, then for all $x\in \Omega$
\begin{equation}
  |u(x)| \le  C_1\|f\|_{L^\infty (\Omega)}\delta(x)^s\,,
\end{equation}
where $C_1$ is a constant depending only on $\Omega$ and $s$. Besides.
 $\left. u/\delta^s\right|_\Omega$ can be continuously extended to $\overline\Omega$, we have $u/\delta^s\in C^\alpha(\Omega)$ and
\begin{equation}
\left \| \frac u {\delta^s} \right \|_{C^\alpha (\Omega)}\le  C_2\|f\|_{L^\infty (\Omega)}
\end{equation}
for some $\alpha>0$ satisfying $\alpha<\min\{s,1-s\}.$ The constants $\alpha$ and $C_1, C_2$ depend only on $\Omega$ and $s$.
\end{enumerate}
\end{proposition}

\begin{remarks}
1) Existence and uniqueness of weak solutions in the classical case $s=1$ is standard. There are also many results about the regularity of the weak solutions with data in Lebesgue spaces. For instance, for $p = 1$ and $n=2$, we have the very sharp results of \cite{fiorenza1998existence}.

2) There are many references to the variational treatment of equations with nonlocal operators, both linear and nonlinear, see \cite{BookMolicaMR3445279}.
\end{remarks}


\subsection{Very weak solutions. A survey on existence, uniqueness and properties}

However, our purpose is  to deal with a larger class of data $f$, which are locally integrable functions, otherwise  as general as possible. A more general definition of solution is necessary. We take an old idea by Brezis (see \cite{Brezis1971}).

\begin{definition} \label{defn:vws Brezis no potential}
	Let $f \in L^1 (\Omega, \delta^s)$. We say that $u$ is a {\sl very weak solutio}n of \eqref{eq:FDE Laplace} if
	\begin{equation} \tag{P$^0_{\textrm{vw}}$} \label{eq:FDE Brezis no potential}
	\begin{dcases}
	u\in L^1 (\Omega),  \\
	u = 0 \textrm{ a.e. } \mathbb R^n \setminus \Omega \textrm{ and } \\
	\int _\Omega u (-\Delta)^s \varphi\,dx  = \int_{\Omega} f \varphi\,dx , \qquad \forall \varphi \in \Xs,
	\end{dcases}
	\end{equation}
	where
	\begin{equation} \label{eq:defn Xs}
	\Xs = \{ \varphi \in \mathcal C^s (\mathbb R^n) : \varphi = 0 \textrm{ in }\mathbb R^n \setminus \Omega \ \textrm{ and } \ (-\Delta)^s \varphi \in L^\infty (\Omega) \} .
	\end{equation}
\end{definition}
This type of solution is also known as very \sl weak solution in the sense of Brezis. \rm

\begin{remarks}  \label{rem:to the definition of weak sol of P0}
	1)  Applying identity \eqref{eq:integration by parts} again we can prove that any weak solution in the sense of Definition     \ref{defn:ws no potential} satisfies our definition of very weak solution.

\medskip

2) This definition allows us to take $f$ to be outside $L^1$, but rather with a weighted integrability condition, $f\in L^1(\Omega; \delta^s) $, which will turn out to be the correct class. About the weight, Ros-Oton and Serra proved in \cite[Lemma 3.9]{Ros-Oton2014} that $\delta^s \in \mathcal C^{\alpha} (\Omega \cap \{ \delta < \rho_0  \})$ for $\alpha = \min \{ s , 1-s\}$ and
	\begin{equation}  \label{eq:delta smooth}
		|(-\Delta)^s \delta^s| \le C_\Omega \qquad \textrm{ in } \Omega \cap \{ \delta < \rho_0  \}.
	\end{equation}
In order to simply the calculations, it is convenient to replace $\delta^s$ by the first eigenfunction of the fractional Laplacian $\varphi_1$, which is positive and smooth everywhere inside $\Omega$ and satisfies  exactly the same boundary behaviour ($\varphi_1 \asymp \delta^s$).

\medskip

3) In $\Xs$ we can only ask for $\mathcal C^s (\Omega)$ smoothness, because, when $\Omega = B_R$, we will want to approximate
	\begin{equation*}
	\begin{dcases}
	(-\Delta)^s \varphi = 1 & B_R \\
	\varphi = 0 & \mathbb R^n \setminus B_R
	\end{dcases}
	\end{equation*}
	which is
	\begin{equation*}
	\varphi (x) = C\left( R^2 - |x|^2 \right)^s
	\end{equation*}
	only of class $\mathcal C^s$. Nonetheless, $\varphi/\delta^s$ can be shown to be smoother. In this case, it is infinitely differentiable, whereas $\varphi$ is not.

\medskip

4) Definition \ref{defn:vws Brezis no potential} corresponds to the notion of {\sl weak dual solution} proposed and used in \cite{Bonforte2015}:
\begin{equation} \label{eq:weak dual solution}
\int_ \Omega u \psi = \int_ \Omega f [(-\Delta)^{s}]^{-1} \psi
\end{equation}
where $[(-\Delta)^{s}]^{-1}$ is the solution operator. We will make a detailed comment about the interpretation of this kind of solution in Section \ref{sec:Green operator viewpoint}.

\medskip
	
	5) By the formula of integration by parts, it is clear that any very weak solution with $f \in L^2(\Omega)$ that is also in $H_0^s (\Omega)$ is a weak solution.
\end{remarks}

Chen and V\'eron \cite{chen+veron2014} seem to have been the first to apply this approach to the fractional case. They proved the following results:
\begin{theorem}[\cite{chen+veron2014}] \label{thm:chen+veron}
	Let $f \delta^s \in L^1( \Omega)$. Then, there exists exactly one very weak solution $u \in L^1 (\Omega)$ of Problem \eqref{eq:FDE Brezis no potential}.  If $f \ge 0$, then $u \ge 0$. Hence, the Maximum Principle holds.
\end{theorem}

\begin{remarks} 1) We point out that the authors also treat semilinear problems of the form $(-\Delta)^s u+ g(u)=f$, and that their work does not apply Green function estimates. The authors work with the more general class of measure data ${\mathcal M}(\Omega, \delta^s)$.

2) The Maximum  Principle allows for the definition of super- and subsolutions that can be useful in getting estimates.

3) A reference to optimal regularity for the fractional case
is \cite{KuusiMR3339179} which also includes nonlinear fractional elliptic problems with $p$-Laplacian type growth. The right-hand side data are locally bounded measures. When $f$ has further regularity the solutions are smooth to different degrees by the representation via the Green kernel \Cref{eq:Green operator}.

4) More properties of the solutions will be examined below and in the study of the Schr\"odinger equation.
\end{remarks}

With the formulation \eqref{eq:FDE Brezis no potential} we can  precisely state the Kato inequality in a general way. The proof for the fractional operator is also due to Chen and V\'eron \cite{chen+veron2014}

\begin{lemma}[Kato's inequality \cite{chen+veron2014}] \label{lem:Kato}
	Let $f \in L^1 (\Omega, \delta^s)$ and $u \in L^1 (\Omega)$ be a solution of \eqref{eq:FDE Brezis no potential}. Then
	\begin{align}
			\int_{ \Omega } |u| (-\Delta)^s \varphi &\le \int_{ \Omega } \sign(u) f \varphi , \\
			\int_{ \Omega } u_+ (-\Delta)^s \varphi &\le \int_{ \Omega } \sign_+(u) f \varphi
	\end{align}
	hold for all $\varphi \in \Xs$, $\varphi\ge0$.
\end{lemma}
	

\subsection{A quantitative lower Hopf principle for data in $L^1 (\Omega, \delta^s)$}

By uniqueness and approximation we easily see that  the preceding solutions admit a  representation via the Green kernel \eqref{eq:Green operator}, and estimates \eqref{eq:Green function estimates} we can prove an adapted lower Hopf inequality. The classical case $s=1$ was first stated in this form in \cite{Diaz+Morel+Oswald:1987}.

\begin{proposition} \label{prop:Hopf}
	 Let $0 \le f \in L^1 (\Omega, \delta^s)$ and let $u \in L^1 (\Omega)$ be the unique very weak solution of \eqref{eq:FDE Brezis no potential}. Then
\begin{equation}\label{eq:Hopf}
	u(x) \ge c \delta(x)^s \int_ \Omega f(y) \delta(y)^s
\end{equation}
a.e. $x \in \Omega$, where $c>0$ depends only on $\Omega$.
\end{proposition}

\begin{proof}
	We first  show that
	\begin{equation} \label{eq:lower bound Green function}
	\mathbb G_s(x,y) \ge c (x) \delta(y)^s
	\end{equation}
	for all $x, y \in \Omega$, and $c > 0$ in $\Omega$. Let $x \in \Omega$. Applying \eqref{eq:Green function estimates}, it is clear that
	\begin{equation}
	c(x) = \inf_{y \in \bar \Omega} \frac{1}{\delta(y)^s |x-y|^{n-2s}} \left( \frac{\delta(x)}{|x-y|} \wedge 1 \right)^s \left( \frac{\delta(y)}{|x-y|} \wedge 1 \right)^s
	\end{equation}
	is reached at some point $y^*$. It is easy to see that $c(x) \ge c \delta(x)^s$ where $c > 0$. Therefore
	\begin{equation}
	u(x) =  \int_{ \Omega } \mathbb G _s (x,y) f(y) dy \ge c \delta(x)^s \int_{ \Omega } f(y) \delta^s (y) dy .
	\end{equation} This completes the proof.
\end{proof}

 The strict positivity of solutions with nonnegative data, in particular the behaviour near the boundary, has been studied in \cite{BFV2018}. There exists also a wide literature for parabolic equations, both linear and nonlinear.

\subsection{$L^1(\Omega; \delta^s)$ as an optimal class of data}
\label{sec:L1 weighted optimal class}

Through the Hopf inequality it is easy to show that this is largest space to look for solutions if we want to keep the class of weak solutions for bounded data and the maximum principle. The nonexistence of solution for such data is a consequence of the following blow-up result.

\begin{proposition} Let $f$ be a nonnegative function such that $f\notin L^1 (\Omega, \delta^s)$, and let $f_k$ be a sequence of  approximations by bounded functions, $f_k\le f$, $f_k\to f$ a.e. in $\Omega$.  Then, $u_k = (-\Delta)^{-s} f_k\to \infty$  in $\Omega$.
\end{proposition}

\begin{proof}  Applying \Cref{prop:Hopf} we have that
\begin{equation}
	u_k(x) \ge   c \delta(x)^s  \int_{ \Omega } f_k(y) \delta^s (y) dy.
\end{equation}
Passing to the limit $k\to\infty$ we would arrive at $u(x)\ge \lim_k u_k(x)=+\infty$.
\end{proof}

\begin{remark} In the limit this optimality can be extended to measure data in the class
${\mathcal M}(\Omega, \delta^s)$.
\end{remark}

\subsection{Traces of very weak solutions and boundary weighted integrability}

The definition of weak solution includes a very clear sense of zero boundary trace since $u\in H^s_0(\Omega)$. However, the definition of very weak solution merely requires  $u\in L^1 (\Omega)$. Clearly, since the  space $L^1 (\Omega)$ does not have a boundary trace operator there is a question about the sense in which the solution $u$ takes null boundary data on $\partial \Omega$. In the classical case $s=1$ some authors have proposed to study local solutions of the Laplace equation inside $\Omega$ that satisfy a generalized $0$ boundary condition of the form $u/\delta \in L^1(\Omega)$. Always for $s=1$, Kufner \cite{Kufner1980} was amongst the first to notice that this kind of singular weights give significant boundary information. We recall that, for $p > 1$, the classical Hardy inequality implies that
\begin{equation}
u \in W^{1,p}_0 (\Omega) \iff \begin{dcases}
u \in W^{1,p}(\Omega) \textrm{ and}\\ \frac{u}{\delta} \in L^p (\Omega).
\end{dcases}
\end{equation}
For $p=1$ this result is no longer true.

The convenience of using the  integral condition ${u}/{\delta} \in L^1 (\Omega)$ as a kind of generalized boundary condition instead of a standard trace condition has been observed recently (see \cite{diaz+gc+rakotoson2017schrodinger}). Moreover, Rakotoson showed in \cite{Rakotoson2012} the following equivalence:
\begin{equation} \label{eq:Rakotoson necessary and suf condition u delta in L1}
\frac u \delta \in L^1 (\Omega) \iff f \delta ( 1 + |\log \delta |) \in L^1 (\Omega) .
\end{equation}

When considering the fractional Dirichlet problem we have found that the appropriate weight is $\delta^s$. To begin with, there is a Hardy inequality for these operators
\begin{proposition}[\cite{Ihnatsyeva2014}] \label{thm:Hardy inequality}
	Let $1 < p <\infty$ and $0 < s < 1$ be such that $sp < n$. Let $\Omega \subset \mathbb R^n$ be an open set, $n \ge 2$, with regular boundary. Then, for, every $u \in C_0^\infty (\Omega)$,
	\begin{equation}
	\int_ {\Omega} \frac{|u(x)|^p}{\delta(x)^{sp}} \le c_0 \int_{ \mathbb R^n } \int_{ \mathbb R^n } \frac{|u(x) - u(y)|^p}{|x-y|^{n+sp}} dy dx,
	\end{equation}
	where $u$ is extended by $0$ outside $\Omega$.
\end{proposition}

Note that for $p=2$ then $\frac n p \ge 1$, so the condition $0 < s < \frac{n}{p}$ is trivial.
Hence, for $u \in H^s_0 (\Omega)$, we know that ${u}/{\delta^s} \in L^2$.

In order to present our results we need to introduce a new test function:
\begin{equation}
\begin{dcases}
(-\Delta)^s \varphi_{\delta} = \frac{1}{\delta^s} & \Omega \\
\varphi_\delta = 0 & \Omega^c.
\end{dcases}
\end{equation}
We prove the following theorem:

\begin{proposition} \label{prop:nec and suf condition for u over deltas in L1}
	Let $f \delta^s \in L^1 (\Omega)$ and let $u \in L^1 (\Omega)$ be the solution of \eqref{eq:FDE Brezis no potential}. Then, $u/\delta^s \in L^1(\Omega)$ if and only if $f \varphi_{ \delta } \in L^1 (\Omega)$.
\end{proposition}

\noindent {\bf Remark.} The difficulty with this function is that, since $1/\delta^s \notin L^\infty$, we know $\varphi_\delta \notin \Xs$.

\begin{proof}  Let us consider the auxiliary functions
	\begin{equation}
		\begin{dcases}
		(-\Delta)^s \varphi_{\delta,k} = \min\left \{ \frac{1}{\delta^s}, k \right \} & \Omega \\
		\varphi_{\delta,k} = 0 & \Omega^c.
		\end{dcases}
	\end{equation}
	Then
	\begin{equation}
	\int_ {\Omega} u  \min\left \{ \frac{1}{\delta^s}, k \right \} = \int_{ \Omega } u (-\Delta)^s \varphi_{ \delta, k } = \int_{ \Omega } f \varphi_{ \delta, k }.
	\end{equation}
	We will prove the case $f \ge 0$, and the sign changing case follows directly. Since $f \ge 0$ then $u \ge 0$. It is clear that $\varphi_{\delta,k}$ is a nondecreasing sequence, the limit of which is $\varphi_ \delta$. Since $f, u \ge 0$, by the Monotone Convergence Theorem
	\begin{equation}
	\int_ {\Omega} \frac{|u|}{\delta^s} = \int_ {\Omega} \frac{u}{\delta^s} = \lim_{k \to \infty} \int_ {\Omega} u  \min\left \{ \frac{1}{\delta^s}, k \right \}  = \lim_{k \to \infty } \int_{ \Omega } f \varphi_{ \delta, k }  = \int_{ \Omega } f \varphi_{ \delta } = \int_{ \Omega } |f| \varphi_{ \delta }.
	\end{equation}
	One integral is finite if and only the other integral is finite.
\end{proof}
We can characterize the behaviour of $\varphi_{ \delta }$ near the boundary:
\begin{lemma} \label{lem:phi delta estimates}
	There exist constants, $c, C > 0$ such that
	\begin{equation}
	c \delta^s |\log \delta| \le \varphi _ \delta \le C \delta^s (1 + |\log \delta|).
	\end{equation}
\end{lemma}

\noindent This result is technical but simple, we give  the proof in \Cref{sec:phi delta estimates}. For $s=1$ the result is due to Rakotoson \cite{Rakotoson2018personal+communication}.

Through this estimate, we provide an extension of \eqref{eq:Rakotoson necessary and suf condition u delta in L1} to the fractional case:
\begin{proposition}[Necessary and sufficient condition for $ u /{\delta^s} \in L^1 (\Omega)$] Let $f \delta^s \in L^1 (\Omega)$ and $u \in L^1 (\Omega)$ be the very weak solution of \eqref{eq:FDE Brezis no potential}. Then,
	\begin{equation} \label{eq:equivalence u deltas if V is 0}
	\frac u {\delta^s} \in L^1 (\Omega) \iff f \delta^s ( 1 + |\log \delta |) \in L^1 (\Omega).
	\end{equation}
	
\end{proposition}

\begin{proof}
	We will only give the proof for $f \ge 0$. Then $u \ge 0$.
	\begin{equation}
	\int_ \Omega \frac{u}{\delta^s} = \int_ \Omega f \varphi_{ \delta } .
	\end{equation}
	Thus
	\begin{equation}
	c_1 \int_\Omega f \delta^s |\log \delta| \le \int_ \Omega \frac u {\delta^s } \le c_2 \int_ \Omega f \delta ^s (1 + | \log \delta |) .
	\end{equation}
	The $\impliedby$ part is then proved.
	
	On the other hand, if $ u /\delta^s \in L^1 (\Omega)$, then $f \delta^s |\log \delta| \in L^1 (\Omega)$. For the first eigenfunction of $(-\Delta)^s$ it is well known that $c_1 \delta^s \le \varphi_1 \le c_2 \delta^s$.  Hence,
	\begin{equation}
	c_1 \int_\Omega f \delta ^s \le \int_ \Omega f \varphi_1 = \int _ \Omega u (-\Delta)^s \varphi_1 = \lambda_1 \int_ \Omega  u \varphi_1 \le \left \| \frac u {\delta^s} \right \|_{L^1 } \| \varphi_1 \delta^s \|_{L^\infty}
	\end{equation}
	Adding both computations, the $\implies$ part of the theorem is proved.
\end{proof}


\subsection{A note on local very weak solutions}

	In this theory, it is natural to define local solutions, if we are not concerned with the boundary information:
	\begin{definition}
		We say that $u$ is a very weak \emph{local} solution if
		\begin{equation} \label{eq:FDE local no potential} \tag{P$^0_{\textrm{loc}}$}
		\begin{dcases}
		u\in L^1 (\Omega),  \\
		u = 0 \textrm{ a.e. } \mathbb R^n \setminus \Omega \textrm{ and } \\
		\int _\Omega u (-\Delta)^s \varphi\,dx  = \int_{\Omega} f \varphi\,dx , \qquad \forall \varphi \in \Xs \cap C_c (\Omega).
		\end{dcases}
		\end{equation}	
	\end{definition}
	Notice that the difference with \eqref{eq:FDE Brezis no potential} lies in space where the test functions are taken. It is clear that, even with the extra requirement $u \in H^s (\Omega)$, there is no uniqueness of solutions. The reason why \eqref{eq:FDE Brezis} has uniqueness of solutions is the fact, in the very weak formulation, that the test function ``sees'' the boundary information. This is due to the integration by parts formula.
	
	Due to the integration  by parts formula and density, any solution \eqref{eq:FDE local no potential} such that $u \in H^s _0 (\Omega)$ is a solution of \eqref{eq:FDE Brezis no potential}. This raises the question: how much extra information does one need to show uniqueness of \eqref{eq:FDE local no potential}. In this section, we will show that
	\begin{equation} \label{eq:u deltas is in L1}
	\frac u { \delta ^s } \in L^1 (\Omega)
	\end{equation}
	 is sufficient information (i.e., there is, at most, one solution of \eqref{eq:FDE local no potential} such that \eqref{eq:u deltas is in L1} holds). \Cref{prop:nec and suf condition for u over deltas in L1} shows that this \eqref{eq:u deltas is in L1} is only possible if $f \varphi_ \delta \in L^1 (\Omega)$, and in this case it is always true. This produces and equivalent formulation of \eqref{eq:FDE Brezis} when $f \varphi_ \delta \in L^1 (\Omega)$.


 \subsubsection{A lemma of approximation of test functions}

	Since the only difference between \eqref{eq:FDE Brezis no potential} and \eqref{eq:FDE local no potential} is the space of test functions, let us study further these spaces. It is clear that one way to pass from \eqref{eq:FDE local no potential} to \eqref{eq:FDE Brezis no potential} will be to select, for each $\varphi \in \Xs$ a sequence $\varphi_k \in \Xs \cap \mathcal C_c^\infty(\Omega)$ such that $\varphi_k \to \varphi$. This convergence must be good enough to preserve the equation. In this direction we introduce the following cut-offs.

Let $\eta$ be a $C^2 (\mathbb R)$ function such that $0 \le \eta \le 1$ and
\begin{equation}
\eta (t) = \begin{cases}
0 & t \le 0 , \\
1 & t \ge 2.
\end{cases}
\end{equation}
We define the functions
\begin{equation}
\eta_\ee (x) = \eta \left(  \frac{{\varphi_1 (x)} - \ee^s}{\ee^s}   \right).
\end{equation}
where $\varphi_1$ is the first eigenfunction of $(-\Delta)^s$. Notice that $\varphi_1 \asymp \delta^s$. We prove  the following approximation result:

\begin{lemma}  \label{lem:approximation of test functions}
	For $\varphi \in \Xs$ we have that $\eta_\ee\varphi\in \Xs\cap C_c(\Omega)$ and
	\begin{align} \label{eq:approximation of test functions}
	\delta^s (-\Delta)^s (\varphi \eta_\ee ) & \rightharpoonup \delta^s (-\Delta)^s \varphi \\
	\frac{\varphi \eta_\ee}{\delta^s} &\rightharpoonup \frac{\varphi}{\delta^s} \label{eq:approximation of test function over deltas}
	\end{align}
	in $L^\infty$-weak-$\star$ as $\ee \to 0$.
\end{lemma}

To prove Lemma \ref{lem:approximation of test functions} we can use the following decomposition:
\begin{theorem}[Eilertsen formula (see \cite{Eilertsen2000})] Let $u,v \in \mathcal C_0^\infty (\mathbb R^n)$ and $ 0 < s < 1$. Then
	\begin{equation} \label{eq:Eilertsen formula}
	(-\Delta)^s (uv) (x)=  u (x)(-\Delta)^s v (x)+  v(x) (-\Delta)^s u (x)  - A_s \int_{\mathbb R^n} \frac{ (u(x)- u(y) ) (v (x) - v (y))}{|x-y|^{n+2s}} dy,
	\end{equation}
	where $A_s \asymp s(1-s)$.
\end{theorem}
The difficult term will be the first one.

\begin{lemma} There exists a constant $C$ independent of $\ee$ such that
	\begin{equation}
	| \delta^{2s} (x) ( -\Delta )^s \eta_\ee (x) | \le C.
	\end{equation}
\end{lemma}
\begin{proof}
	In order to use \cite[Lemma 2.3]{chen+veron2014}
	we write $\eta_\ee (x) = \gamma_\ee (\varphi_1(x))$, where
	\begin{equation}
	\gamma_\ee (t) = \eta \left(  \frac{t - \ee^s}{\ee^s}   \right).
	\end{equation}
	Due to \eqref{eq:delta smooth}, we have that $\eta_\ee \in \Xs$ and, and for every $x \in \overline \Omega$, there exists $z_x \in \overline \Omega$
	\begin{align}
	(-\Delta)^s \eta_\ee (x) & =  (-\Delta)^s (\gamma_\ee \circ {\varphi_1}) (x) \\
	&= \gamma_\ee' ({\varphi_1}(x)) (-\Delta)^s {\varphi_1} (x) + \frac{\gamma''({\varphi_1}(z_x))}{2s} \int_{ \Omega } \frac{|{\varphi_1}(x)-{\varphi_1}(y)|^2}{|x-y|^{n+2s}}dy\\
	&=\frac{1}{\ee^s} \eta'\left(  \frac{{\varphi_1}(x) - \ee^s}{\ee^s}  \right) { \lambda_1 \varphi_1} (x) \\
	&\quad + \frac{\eta'' \left(  \frac{{\varphi_1}(z_x) - \ee^s}{\ee^s}  \right)}{2 s\ee^{2s}} \int_{ \Omega } \frac{|{\varphi_1}(x)-{\varphi_1}(y)|^2}{|x-y|^{n+2s}}dy.
	\end{align}
	From this, applying {the regularity of $\varphi_1$}
	\begin{equation} \label{eq:global bound of eta ee}
	| (-\Delta)^s \eta_\ee (x) | \le \ee ^{-2s}, \qquad \forall x \in \overline \Omega.
	\end{equation}
	
	Consider, in particular, that $\delta(x) \ge 3 \ee$. Let $ d = d ( x, \{   \ee < \delta < 2 \ee \} ) > 0$, so $ \eta_\ee (y) = 1 = \eta_\ee (x)$ if $|x-y| \le d$. We have that
	\begin{align}
	| (-\Delta)^s \eta_\ee (x) | &= \left |   \int _{\mathbb R^n } \frac{\eta_\ee (x) - \eta_\ee (y)}{|x-y|^{ n + 2s } } dy \right | \\
	&=\left |   \int _{|x-y| > d } \frac{\eta_\ee (x) - \eta_\ee (y)}{|x-y|^{ n + 2s } } dy \right | \\
	&\le    \int _{|x-y| > d } \frac{2}{|x-y|^{ n + 2s } } dy  \\
	&= C \int_{d}^{+\infty} \frac{1}{r^{n+2s} }r^{n-1} dr \\
	&= \frac{1}{d^{2s}}.
	\end{align}
	Since $d \ge \delta(x) - 2\ee \ge \delta(x)/3 $. We have that
	\begin{equation}
	| (-\Delta)^s \eta_\ee (x) | \le \frac{C}{\delta(x)^{2s}}, \qquad \forall x \textrm { such that } \delta (x) \ge 3 \ee.
	\end{equation}
	On the other hand,
	\begin{equation}
	| \delta^{2s}(x)(-\Delta)^s \eta_\ee (x) | \le 3^{2s} \ee ^{2s} C \ee^{-2s} = C \qquad \qquad \forall x \textrm { such that } \delta (x) \le 3 \ee.
	\end{equation}
	This proves the result.
\end{proof}

\begin{lemma} For all $\varphi \in \Xs$
	\begin{equation}
	\left| \delta^{s} (x) \int_{\mathbb R^n} \frac{ (\varphi(x)- \varphi(y) ) (\eta_\ee (x) - \eta_\ee (y))}{|x-y|^{n+2s}} dy\right| \le C,
	\end{equation}
	where $C$ does not depend on $\ee$.
\end{lemma}

\begin{proof}
	For any $x \in \overline \Omega$
	\begin{align}
	\left| \int_{\mathbb R^n} \frac{ (\varphi(x)- \varphi(y) ) (\eta_\ee (x) - \eta_\ee (y))}{|x-y|^{n+2s}} dy\right| & = \left| \int_{\mathbb R^n} \dfrac{ (\varphi(x)- \varphi(y) ) \eta'\left(  \frac{{\varphi_1}(z_y) - \ee^s}{\ee^s}   \right)\frac{ {\varphi_1} (x) - {\varphi_1}(y)}{\ee^s}}{|x-y|^{n+2s}} dy \right| \nonumber \\
	&\le C \ee^{-s} \left| \int_{\mathbb R^n} \dfrac{ (\varphi(x)- \varphi(y) ) ({ {\varphi_1} (x) - {\varphi_1}(y)})}{|x-y|^{n+2s}} dy \right|  \nonumber \\
	& \le C \ee^{-s}. \label{eq:product gradients general}
	\end{align}
	
	We compute, for $\delta(x) \ge 3 \ee$
	\begin{align}
	\left| \int_{\mathbb R^n} \frac{ (\varphi(x)- \varphi(y) ) (\eta_\ee (x) - \eta_\ee (y))}{|x-y|^{n+2s}} \right| & \le \left(   \int_{\mathbb R^n} \frac{ |\varphi(x)- \varphi(y) |^2 }{|x-y|^{n+2s}}   \right)^{\frac 1 2} \left(  \int_{\mathbb R^n} \frac{ |\eta_\ee (x) - \eta_\ee (y)|^2}{|x-y|^{n+2s}}  \right)^{\frac 1 2} \\
	& \le C \|\varphi\|_{H^s (\mathbb R^n)} \left(  \int_{|x-y|\ge d} \frac{1}{|x-y|^{n+2s}}\right)^{\frac 1 2} \\
	& \le \frac{C \|\varphi\|_{H^s(\mathbb R^n)}}{d^s}.
	\end{align}

	From this estimate and \eqref{eq:product gradients general} we conclude, as before, the result.
\end{proof}
\begin{proof}[Proof of \Cref{lem:approximation of test functions}]
	We write Eilertsen's formula in our case
	\begin{equation} \label{eq:Eilertsen formula particular}
	(-\Delta)^s (\varphi \eta_\ee ) (x)=    \eta_\ee (x) (-\Delta)^s \varphi (x) + \varphi(x) (-\Delta)^s \eta_\ee (x) - A_s \int_{\mathbb R^n} \frac{ (\varphi(x)- \varphi(y) ) (\eta_\ee (x) - \eta_\ee (y))}{|x-y|^{n+2s}} dy.
	\end{equation}
	We have proven that the second and third term are bound when multiplied by $\delta^{2s}(x)$ and $\delta^s(x)$, respectively. They converge pointwise to $0$. Hence, up to a subnet,
	\begin{align}
	\delta^{2s} (x) (-\Delta)^s \eta_{\ee} & \rightharpoonup  0, \\
	\delta^{s} (x) \int_{\mathbb R^n} \frac{ (\varphi(x)- \varphi(y) ) (\eta_\ee (x) - \eta_\ee (y))}{|x-y|^{n+2s}} dy & \rightharpoonup 0 \qquad \textrm{ in } L^\infty\textrm{-weak-}\star.
	\end{align}
	Since $\varphi \in \Xs$ we know that $\frac{ \varphi }{\delta^s} \in L^\infty$. Hence, up to a subnet,
	\begin{equation}
	\delta^s \varphi (-\Delta)^s \eta_\ee = \frac{\varphi}{\delta^s} \delta^{2s} (-\Delta)^s \eta_\ee \rightharpoonup 0 \textrm{ in } L^\infty\textrm{-weak-}\star.
	\end{equation}
	On the other hand $ \eta_{\ee}$ is bounded, and converges pointwise to $1$. Hence, up to a subnet,
	\begin{equation}
	\eta_{\ee} \rightharpoonup 1 \qquad \textrm{ in } L^\infty\textrm{-weak-}\star.
	\end{equation}
	Thus, up to a subnet,
	\begin{equation}
	\delta^s \eta_{\ee} \rightharpoonup \delta^s \qquad \textrm{ in } L^\infty\textrm{-weak-}\star.
	\end{equation}
	All the above are bounded net, such that every subnet have the same limit. All nets converge. This proves \eqref{eq:approximation of test functions}. \\
	
	\noindent On the other hand
	\begin{equation}
	\frac{\varphi \eta_{\ee}}{\delta^s} = \frac{\varphi }{\delta^s} \eta_{\ee} \rightharpoonup \frac{\varphi }{\delta^s} \qquad \textrm{ in } L^\infty\textrm{-weak-}\star,
	\end{equation}	
	proving \eqref{eq:approximation of test function over deltas}. This concludes the proof.
\end{proof}


\subsubsection{A local solution which is integrable with a suitable boundary weight is a v.w.s.}
	
	\begin{proposition} \label{prop:no potencial:local solution and weighted implies vws}
		Any solution of \eqref{eq:FDE local no potential} such that $\frac u {\delta^s} \in L^1 (\Omega)$ is a solution of \eqref{eq:FDE Brezis no potential}.
	\end{proposition}
	
	\begin{proof}
		Let $\varphi \in \Xs$. Consider an approximation $\eta_{\frac 1 k} \varphi \in \Xs \cap C_c (\Omega)$. Since ${u}/{\delta^s} \in L^1 (\Omega)$ then
		\begin{align}
		\int_ \Omega u ( -\Delta)^s (\eta_{\frac 1 k} \varphi) &= \int_ \Omega f \varphi \eta_{\frac 1 k}, \\
		\int_ \Omega \frac{u}{\delta^s} \delta^s ( -\Delta)^s (\eta_{\frac 1 k} \varphi) &= \int_ \Omega f \delta^s \frac {\varphi \eta_{\frac 1 k}}{\delta^s}.
		\end{align}
		By passing to the limit applying \Cref{lem:approximation of test functions}
		\begin{equation}
		\int_ \Omega u ( -\Delta)^s  \varphi  = \int_ \Omega f \varphi.
		\end{equation}
		This proves the result.
	\end{proof}
	As a corollary of the uniqueness of \eqref{eq:FDE Brezis no potential} we have the following
	\begin{proposition}
		There is, at most, one solution of \eqref{eq:FDE local no potential} such that $ u /{\delta^s} \in L^1 (\Omega)$.
	\end{proposition}

\paragraph{Summary}  It is obvious that $
	\eqref{eq:FDE Brezis no potential}  \implies  \eqref{eq:FDE local no potential}$.
	\Cref{prop:nec and suf condition for u over deltas in L1} states that
	\begin{eqnarray}
	\begin{dcases}
	\eqref{eq:FDE Brezis no potential} \textrm{ and}\\
	f \varphi_\delta \in L^1(\Omega)
	\end{dcases}
	\implies
	\begin{dcases}
	\eqref{eq:FDE Brezis no potential} \textrm{ and}\\
	\frac{u}{\delta^s} \in L^1(\Omega)
	\end{dcases}
	\end{eqnarray}
	\Cref{prop:no potencial:local solution and weighted implies vws} states:		
	\begin{eqnarray}
	\eqref{eq:FDE Brezis no potential}  \impliedby  \begin{dcases}
	\eqref{eq:FDE local no potential} \textrm{ and}\\
	\frac{u}{\delta^s} \in L^1(\Omega)
	\end{dcases}
	\end{eqnarray}
	Combining both facts:	
	\begin{eqnarray}
	\begin{dcases}
	\eqref{eq:FDE Brezis no potential} \textrm{ and}\\
	f \varphi_\delta \in L^1(\Omega)
	\end{dcases}
	\iff
	\begin{dcases}
	\eqref{eq:FDE local no potential} \textrm{ and}\\
	\frac{u}{\delta^s} \in L^1(\Omega).
	\end{dcases}
	\end{eqnarray}
\normalcolor

Finally, let us state the following comparison result.

 	\begin{proposition}[Comparison principle] \label{lem:comparison principle}
	Assume that
	\begin{equation} \label{eq:maximum principle}
	\begin{dcases}
	\int_{ \Omega } u (-\Delta)^s \varphi \le 0 \qquad \forall 0 \le \varphi \in \Xs \cap C_c (\Omega) ,\\
	\frac u {\delta^s} \in L^1(\Omega) .
	\end{dcases}
	\end{equation}
	Then $u \le 0$.
\end{proposition}

\begin{proof}[Proof of \Cref{lem:comparison principle}] \label{proof of lem:comparison principle}
	Let $\varphi \in \Xs$. Take $\varphi_k  = \varphi \eta_{\frac 1 k} \in \Xs \cap C_c (\Omega)$. Then, by \Cref{lem:approximation of test functions},
	\begin{equation}
	0 \ge \int_{ \Omega } u (-\Delta)^s \varphi_k = \int_{ \Omega } \frac{u}{\delta^s} \delta^s (-\Delta)^s \varphi_k \to \int_{ \Omega } \frac{u}{\delta^s} \delta^s (-\Delta)^s \varphi = \int_{ \Omega } u (-\Delta)^s \varphi.
	\end{equation}
	Hence, taking the test function solution of
	\begin{equation}
	\begin{dcases}
	(-\Delta)^s \varphi = \sign_+ u & \Omega, \\
	\varphi = 0 & \Omega^c,
	\end{dcases}
	\end{equation}
	we deduce that
	\begin{equation}
	\int_ \Omega u_+ \le 0.
	\end{equation}
	Hence $u_+ = 0$. This completes the proof.
\end{proof}

\begin{remark}
	We notice that for $0 < s < 1$ we have  $\delta^{-s} \in L^1(\Omega)$, unlike for $s = 1$. This makes $ u /{\delta^s} \in L^1$ not entirely a ``boundary condition''. For $s = 1$, \eqref{eq:maximum principle} was shown in \cite{diaz+gc+rakotoson2017schrodinger}. On the other hand, in the limit $s = 0$ it says
	\begin{equation}
	\begin{dcases}
	u = (-\Delta)^0 u \le 0 & \Omega \\
	u = \frac{u}{\delta^0} \in L^1 (\Omega)
	\end{dcases}
	\end{equation}
	The second item gives no information, but still the result is trivially true for $s = 0$, and all the information comes from the operator. For $s = 1$ most of the information came from the integral condition. For the interpolation $0 < s < 1$, the responsibility needs to be shared.
\end{remark}
To give an intuition on how much more information the fractional Laplacian $s<1 $ has with respect to the classical Laplacian ($s=1$) we provide the following example:
\begin{example}
	Let $u_c = c \chi_\Omega$ (where $\chi$ is the characteristic function). Then:
	\begin{enumerate}
		\item[(i)] if $c < 0$ then $(-\Delta)^s u_c (x) < 0 $ in $\Omega$,
		\item[(ii)] if $c > 0$ then $(-\Delta)^s u_c (x) > 0$ in $\Omega$.
	\end{enumerate}
For the proof note that for every $x\in\Omega$ we have
$$
(-\Delta)^s u_c (x)=c\int_{\Omega^c}\frac{dy}{|x-y|^{n+2s}}\,.
$$
Both signs are reversed for  $x\in \Omega^c$. This property is obviously false for $s = 1$.
\end{example}


\subsection{Accretivity}

In the study of evolution equations associated to elliptic operators the property of \sl accretivity \rm plays an important role since it can be used as  a basic tool in the solution of associated parabolic problems and the generation of the  corresponding semigroups, \cite{brezis2010functional}.
 We say that a (possibly unbounded or nonlinear) operator $A$ acting in a Banach space $X$ is accretive if for every $u_1, u_2\in D(A)$, the domain of the operator $D(A)\subset X$, and every $\lambda>0$ we have
\begin{equation}\label{accret}
\| u_1-u_2\|_X\le \|f_1-f_2\|_X\,,
\end{equation}
where $f_i=u_i+ \lambda Lu_i$, $i=1,2$. This is a contractivity property. Moreover, an accretive operator is called  $m$-accretive if the problem $f=u+ \lambda Lu$ an be solved for every $f\in X$ and every $\lambda>0$ and the solution
$u$ lies in $D(A)$. The Crandall-Liggett Theorem \cite{crandall1971generation} implies that, when $L$ is an $m$-accretive operator in a Banach space $X$, we can solve the evolution problem
$$
\partial_t u(t)+ Lu(t)=f(t)
$$
for every initial data $u(0)\in X$ for every $f\in L^1 (0,\infty; X)$, and find a unique generalized solution $u\in C([0,\infty); X)$ that solves this initial-value problem in the so-called mild sense.

A further concept is $T$-accretivity, that incorporates the maximum principle and applies to ordered Banach spaces, like spaces of real functions. It reads
\begin{equation}\label{Taccret}
\| ( u_1-u_2)_+\|_X\le \|(f_1-f_2)_+\|_X.
\end{equation}
under the same assumptions as in \eqref{accret}.

The  results of the preceding subsections allow to prove the first part of the following
statement.

\begin{proposition}
	The fractional Laplacian operator $L=(-\Delta)^s$ is $m$-$T$-accretive in the space $L^1(\Omega)$ and also in the spaces $L^1(\Omega; \phi)$ for all positive weights  $\phi\in X^s$, such that $(-\Delta)^s \phi\ge 0$. The restricted Laplacian operator is also $m$-$T$-accretive in the spaces $L^p(\Omega)$ with $1<p\le \infty$.
\end{proposition}

For the accretivity in $L^1(\Omega; \phi)$ we have to check that
$$
\int_{\Omega} Lu \sign(u)\phi\,dx=\int_{\Omega} L|u|\phi\,dx=
\int_{\Omega} |u|L\phi\,dx\ge 0,
$$
where we use Kato's inequality, see \Cref{lem:Kato}, and the symmetry implied by \eqref{double}.

The last statement for finite $p>1$ admits an easy proof that uses the Stroock-Varopoulos  inequality for weak solutions that we quote from \cite{dPQRV-MR2954615}, Lemma 5.1:

\begin{lemma}[Stroock-Varopoulos' inequality] \label{lem:S-V}
Let $0<s<1$, $p>1$.
Then
\begin{equation}\label{eq:strook.varopoulos}
\int_{\mathbb{R}^N}(|v|^{p-2}v)(-\Delta)^{s} v \ge
\frac{4(p-1)}{p^2}\int_{\mathbb{R}^N}\left|(-\Delta)^{s/2}|v|^{p/2}\right|^2
\end{equation}
for all $v\in L^p(\mathbb{R}^n)$ such that $(-\Delta)^{s}v\in
L^p(\mathbb{R}^n) $.
\end{lemma}
This is done for functions defined in $\mathbb{R}^n$, when working in a bounded domain we recall that $v=0$ outside of $\Omega$. The inequality not only shows that the operator is accretive but it measures its amount in terms of a square norm of the fractional operator of half order.

The application for very weak solutions is obtained by passage to the limit. For $p=\infty$ pass to the limit in the result for finite $p$.


\subsection{Comparison with the class of large solutions}

The theory we have described asks  for zero boundary conditions, but only  in some generalized sense. An immediate extension of our class of solutions is the class of large solutions that has  been studied by \cite{abatangelo2013large}. These solutions blow-up at the boundary,  which is explained by the presence of some singular boundary measure in the weak formulation. See also  \cite{FelmerMR2985500}. The typical example is
$$
u_{1-s}(x) =\frac{c(n,s)}{(1 -|x|^2)^{1-s}} \qquad \mbox{in } B_R(0)
$$
with $u_{1-s}(x) =0$ outside. This function is found in  \cite{bogdanMR2569321}, where it is proved that $u_{1-s}$   satisfies
$$
(-\Delta)^s u_{1-s}=0 \qquad\mbox{pointwise in} \ B.
$$
Since $u_{1-s}(x)/\delta(x)^s\asymp c/\delta(x)$, which is not integrable near the boundary, we are sure that this is a large solution and not a very weak solution of the Dirichlet Problem as in the preceding theory. The divergence $\delta(x)^{-1}$ is just borderline for our class of very weak solutions, and this is another proof of optimality for our theory.
\normalcolor


\section{Schr\"odinger problem with positive potentials}\label{sec.schr1}

 Here we extend previous results by authors in \cite{diaz+gc+rakotoson+temam:2018veryweak,diaz+gc+rakotoson2017schrodinger} dealing with the classical stationary Schr\"odinger equation to fractional operators, i.e. to problem \eqref{eq:FDE}. As a preliminary, we start by the easier case of bounded potentials and functions.

By analogy to Definitions \ref{defn:ws no potential} and \ref{defn:vws Brezis no potential}  we introduce

\begin{definition} \label{defn:weak solution}
	{Let $f \in L^2 (\Omega)$.} A {\sl weak solution} of \eqref{eq:FDE} if a function $u \in H^s_0 (\Omega)$, \normalcolor $Vu\in L^2(\Omega)$, \normalcolor and  such that
	\begin{equation} \tag{P$_{\textrm{w}}$} \label{eq:FDE weak}		
	\int_ {\mathbb R^n } (-\Delta)^{\frac s 2} u ( -\Delta)^{\frac s 2 } \varphi + \int_ \Omega Vu \varphi = \int_{ \Omega } f \varphi
	\end{equation}
	for all $\varphi \in H_0^s (\Omega)$.
{This definition can be extended by asking that $f, Vu \in (H_0^s (\Omega))'$.}
\end{definition}

\begin{definition}
{We assume that $f \in L^1 (\Omega, \delta^s)$.} We say that $u$ is a {\sl very weak solution} of \eqref{eq:FDE}
	if
	\begin{equation} \label{eq:FDE Brezis} \tag{P$_{\textrm {vw}}$}
	\begin{dcases}
	u\in L^1 (\Omega),  \\
	u = 0 \textrm{ a.e. } \mathbb R^n \setminus \Omega \textrm{ and } \\
{Vu \delta^s \in L^1 (\Omega)}, \\
	\int _\Omega u (-\Delta)^s \varphi  + \int_ \Omega V u \varphi = \int_{\Omega} f \varphi , \qquad \forall \varphi \in \Xs,
	\end{dcases}
	\end{equation}
	where $\Xs$ is given by \eqref{eq:defn Xs}.
\end{definition}

As seen in the previous section, the concept of weak solution will not be sufficient to solve the Dirichlet Problem with general data. Moreover,  through \Cref{prop:fractional integration by parts} it is trivial to show that

\begin{lemma} \label{lem:weak implies very weak}
	If $u \in H_0^s (\Omega)$ is a weak solution of \eqref{eq:FDE} {with $f \in L^2 (\Omega)$} in the sense of \eqref{eq:FDE weak}, then it is a very weak solution of \eqref{eq:FDE} in the sense of \eqref{eq:FDE Brezis}.
\end{lemma}

\begin{remark}
		The converse implication, which we indicated as true for \eqref{eq:FDE Laplace} in \Cref{rem:to the definition of weak sol of P0}.3), escapes the interest of this paper. However, since $u \in H_0^s (\Omega)$ and $f \in L^2(\Omega)$ then, it seems natural, although it requires a rigorous proof, that $Vu = f - (-\Delta)^s u \in (H_0^s (\Omega))'$. Even if we do not prove that $Vu \in L^2 (\Omega)$, this would be enough to say that $u$ is a weak solution (in a natural sense).
	\end{remark}

The following result confirms that the class of very weak solutions is not too general

\begin{theorem} \label{thm:uniqueness} 
	Let $f \in L^1 (\Omega, \delta^s)$. There is, at most, one solution of \eqref{eq:FDE Brezis}.
\end{theorem}
\begin{proof}
	Let $u_1, u_2$ be two solutions. Let $u = u_1 - u_2$. Therefore,
	\begin{equation}
			\int _\Omega u (-\Delta)^s \varphi  =- \int_ \Omega V u \varphi , \qquad \forall \varphi \in \Xs,
	\end{equation}
	Therefore, through \Cref{lem:Kato} and Kato's inequality \Cref{Kato} we have that
	\begin{equation}
		\int _\Omega |u| (-\Delta)^s \varphi  \le 0, \qquad \forall \ 0 \le \varphi \in \Xs,
	\end{equation}
	In particular, $|u| \le 0$. This completes the proof.
\end{proof}


\subsection{The case $V \in L^\infty (\Omega)$}

We now address the question of existence. The simplest case concerns bounded potentials.
When $V \in L^\infty ( \Omega)$, $(-\Delta)^s + V$ is a self-adjoint operator in $L^2(\Omega)$, as $(-\Delta)^s$, which has a positive first eigenvalue  $\lambda_{1} > 0$.  It is easy to see, through the Lax-Milgram theorem, that there exists a unique weak solution.
		
		When $V, f \in L^\infty( \Omega)$ we can apply the regularity estimates in \Cref{prop:Ros-Oton} by bootstrapping. For a fixed solution, we define $g = f - Vu \in L^\infty (\Omega)$, and we know that $u \in \mathcal C^s (\bar \Omega)$.

\begin{lemma} \label{lem:estimates V and f bounded}
	Let $V, f \in L^\infty (\Omega)$. There exists a unique solution $u$ of \eqref{eq:FDE weak}. It satisfies
	\begin{subequations}
	\begin{align}
	\| u \|_{L^1} & \le C \| f \delta^s \|_{L^1}\label{eq:u in L1} \\
	\| V u \delta^s \|_{L^1} &\le C \|f \delta^s \|_{L^1 } \label{eq:V u deltas in L1},
	\end{align}
	\end{subequations}
	where $C$ does not depend on $u$. Furthermore,
		\begin{subequations}
		\begin{align}
		\left \| \frac u {\delta^s} \right \|_{L^1} &\le \|f \varphi_\delta \|_{L^1 }  \label{eq:u deltas in L1}, \\
			\| V u \varphi_ \delta  \|_{L^1} &\le \|f \varphi_\delta \|_{L^1 } \label{eq:V u phi delta in L1}.
		\end{align}
		\end{subequations}
Moreover,
		\begin{subequations}
		\begin{align}
			\| u \|_{L^2 (\Omega)} &\le C \| f \|_{L^2 (\Omega)} \label{eq:u in L2}\\
			\| (-\Delta)^{\frac s 2} u\|_{L^2 (\mathbb R^n)} &\le C \|f \|_{L^2 (\Omega)}. \label{eq:estimate Ls u}
		\end{align}
		\end{subequations}	
	If $f \ge 0$, then $u \ge 0$.
\end{lemma}

	Notice that \eqref{eq:u deltas in L1} and \eqref{eq:V u phi delta in L1} hold with constant $1$. In order for \eqref{eq:u in L1} and \eqref{eq:V u deltas in L1} to also hold with constant $1$ we can choose the first eigenfunction of $(-\Delta)^s$, $\varphi_1$, as a weight.

\begin{proof}[Proof of \Cref{lem:estimates V and f bounded}]\label{proof lem:estimates V and f bounded} The existence of a weak solution $u \in H_0^s (\Omega)$ follows from the Lax-Milgram theorem. It is a very weak solution. The fact that, if $f \ge 0$, then $u \ge 0$ follows as for the $(-\Delta)^s$ operator.
		
To compute the estimates, we start by considering $f \ge 0$. Then $u \ge 0$.
		By considering as test function the unique solution of problem
		\begin{equation}
		\begin{dcases}
		(-\Delta)^s \varphi_{0} = 1 & \Omega \\
		\varphi_0 = 0 & \Omega^c .
		\end{dcases}
		\end{equation}	
		From the representation formula \eqref{eq:Green operator} and \eqref{eq:Green function estimates} we know that $\varphi_0 \ge c \delta^s$ and, from the results in \cite{Ros-Oton2014}, that $\varphi_0 \in \Xs$. Therefore
		\begin{align}
		\int_ \Omega u + \int _ \Omega V u \delta^s  \le \int_ \Omega u + \frac{1}{c} \int_{\Omega} Vu \varphi_0  \le C \int_{ \Omega }(  u + V u \varphi_0 )  & \le C\int_{ \Omega }f \delta^{s} \frac{ \varphi_0 } {\delta^s}  \le C \left\|\frac{ \varphi_0 } {\delta^s}\right \|_{L^\infty} \|f\delta^{s}\|_{L^1}
		\end{align}
		so \eqref{eq:u in L1} and \eqref{eq:V u deltas in L1} hold.
		Using $\varphi_\delta$ as a test function in the very weak formulation
		\begin{align}
		0 \le \int_ \Omega \frac u {\delta^s}  + V u \varphi_\delta \le   \int_ \Omega u (-\Delta)^s \varphi_\delta + V u \varphi_\delta &= \int_{\Omega} f \varphi_\delta.
		\end{align}
Hence,
		\begin{align}
		\left\| \frac{ u }{\delta^s}  \right\|_{L^1} \le \|f \varphi_ \delta \|_{L^1}.
		\end{align}
		
			To obtain \eqref{eq:u in L2} we can take as a test function in the very weak formulation the solution of
			\begin{equation}
				\begin{dcases}
					(-\Delta)^s \varphi = u & \Omega, \\
					\varphi = 0 & \partial \Omega.
				\end{dcases}
			\end{equation}
			It is clear that $\varphi \ge 0$ and, due to Green kernel estimates, $\| \varphi \|_{L^2 } \le C \| u \|_{L^2}$. Thus
			\begin{equation}
				\int_{ \Omega } u^2 \le \int_{ \Omega } u^2 + \int_{\Omega} Vu \varphi = \int_{ \Omega } f \varphi \le \| f \|_{L^2} \| \varphi \|_{L^2} \le C \| f \| _{L^2 } \|u \|_{L^2}.
			\end{equation}
			This concludes \eqref{eq:u in L2}.
			Since $u$ is a weak solution, \eqref{eq:estimate Ls u} can be obtained by using $u$ as a test function
			\begin{equation}
				\int_ {\mathbb R^n} |(-\Delta)^{\frac s 2} u|^2 \le \int_ {\mathbb R^n} |(-\Delta)^{\frac s 2} u|^2 + \int_{ \Omega } Vu ^2 = \int_{ \Omega } f u \le \|f \|_{L^2} \|u\|_{L^2} \le C \|f \|_{L^2}^2.
			\end{equation}
Finally, if $f$ changes sign, we can decompose it as $f = f^+ - f^-$, and apply twice the previous result to complete the proof.
	\end{proof}

\begin{remark}
	Due to \Cref{thm:Hardy inequality} applied to the case $p =2$, we know that $u \in H^s_0 (\Omega) \mapsto {u}/{\delta^s} \in L^2 (\Omega)$ is well-defined and continuous. Hence, for $0 \le V \le C \delta^{-2s}$ the following bilinear map is continuous
	\begin{eqnarray}
		H^s_0(\Omega) \times H^s_0 (\Omega) &\longrightarrow & \mathbb R \\
		(u, \varphi) & \longmapsto & \int_ {\Omega} V u \varphi
	\end{eqnarray}
	because
	\begin{equation}
		\left |\int_ {\Omega} V u \varphi \right| = \left| \int_{ \Omega } V \delta^{-2s} \frac{u}{ \delta^s } \frac{ \varphi } {\delta^s} \right| \le C \left \| \frac{u}{\delta^s} \right \|_{L^2 (\Omega) }   \left \| \frac{\varphi }{\delta^s} \right \|_{L^2 (\Omega) }  \le C \|u\|_{H^s (\mathbb R^n)} \| \varphi \|_{H^s (\mathbb R^n)}.
	\end{equation}
	Thus, when $0 \le V \le C \delta^{-2s}$ and $f \in L^2 (\Omega)$, we can also use the Lax-Milgram Theorem to show existence and uniqueness of weak solutions.
\end{remark}

\begin{remark} About regularity for weak solutions of  nonlocal Schr\"odinger equations  in an open set of $\mathbb{R}^n$ subject to exterior Dirichlet, recently Fall \cite{Fall.arxiv2017} proves H\"older regularity estimates for general nonlocal operators defined via Dirichlet forms, by symmetric kernels $K(x, y)$ bounded from above and below by $|x-y|^{-(N+2s)}$, $0<s<1$. See also \cite{Frazier2018}.
\end{remark}


\subsection{General potentials $V \in L^1_{loc} (\Omega)$}

We now consider the problem for $0 \le V \in L^1_{loc} (\Omega)$.  For solutions of \eqref{eq:FDE Brezis}, $Vu \delta^s$ could, in principle, not be in $L^1 (\Omega)$. We introduce the following definition
\begin{definition}
	We say that $u$ is a very weak \emph{local} solution of \eqref{eq:FDE} if
	\begin{equation} \tag{P$_{\textrm{loc}}$}
	\label{eq:FDE local}
	\begin{dcases}
	u\in L^1 (\Omega), u = 0 \textrm{ a.e. } \mathbb R^n \setminus \Omega, \\
	{Vu \in L^1_{loc} (\Omega)} \textrm{ and } \\
	\int _\Omega u [(-\Delta)^s \varphi + V \varphi]  = \int_{\Omega} f \varphi , \qquad \forall \varphi \in \Xs \cap C_c (\Omega).
	\end{dcases}
	\end{equation}
\end{definition}
Note that for a local solution, $Vu \delta^s \in L^1 (\Omega)$ does not seem like a natural part of the definition.
It is clear that \eqref{eq:FDE local} is a weaker concept than \eqref{eq:FDE Brezis} because it lacks the information on the boundary, so it cannot produce uniqueness. But it is a very convenient step into existence.\normalcolor


In spaces with traces, solutions of \eqref{eq:FDE local} with trace $0$ are solutions of \eqref{eq:FDE Brezis}. The following theorem shows what a local solution is missing to become a very weak solution

\begin{theorem} \label{thm:equivalence of formulations}
		Let $V \in L^1_{loc} (\Omega)$ and $f \delta^s \in L^1 (\Omega)$. Any solution $u \in L^1 (\Omega)$ {of \eqref{eq:FDE local} such that $Vu \delta^s \in L^1 (\Omega)$ and ${u}/{\delta^s} \in L^1 (\Omega)$} is a solution of \eqref{eq:FDE Brezis}.
\end{theorem}

\begin{proof} \label{proof thm:equivalence of formulations}
	If $Vu \delta^s \in L^1 (\Omega)$, then $g = f - Vu \in L^1 (\Omega, \delta^s)$. Hence, we can apply \Cref{prop:no potencial:local solution and weighted implies vws}.
\end{proof}		

We are ready to state one of the main results of the paper.

\begin{theorem}[Existence theorem] \label{thm:existence}
		Let $V \in L^1_{loc} (\Omega)$ and $f \delta^s \in L^1 (\Omega)$. Then

\noindent {\rm (i)}   There exists a very weak solution of {\eqref{eq:FDE Brezis}}. It satisfies \eqref{eq:u in L1} and \eqref{eq:V u deltas in L1}.

\noindent {\rm (ii)}   If $f \ge 0$, then $u \ge 0$.	
			
\noindent {\rm (iii)}  Furthermore, if $f \varphi_\delta \in L^1 (\Omega)$ then \eqref{eq:u deltas in L1} and \eqref{eq:V u phi delta in L1} hold and, hence,
${u}/{\delta^s} \in L^1 (\Omega)$.

	\noindent {\rm (iv)} Moreover, if $f \in L^2 (\Omega)$, then $u$ is in $H_0^s (\Omega)$.
\end{theorem}

\begin{remark}
		This result extends previous results by the two first authors for the classical case ($s=1$) (see \cite{diaz+gc+rakotoson+temam:2018veryweak,diaz+gc+rakotoson2017schrodinger}) to the fractional case. Furthermore, the argument we provide allows us to improve the results for the classical case. In the present text, we have proved that the definition of \emph{very weak solution in the weighted sense} used in previous papers is not necessary as a concept of solution, but rather as a intermediate step.
	\end{remark}


	\begin{proof}[Proof of \Cref{thm:existence}] We proceed in several steps.
		
	(1)	We start by assuming $f \ge 0$ and bounded. Let $V_k = \min\{V,k\}$.
		Let $u_{k}$ be the solution of
		\begin{equation}
		\begin{dcases}
		(-\Delta)^s u_{k} + V_k u_{k} = f & \Omega, \\
		u = 0 & \Omega^c.
		\end{dcases}
		\end{equation}
We know that
		\begin{align}
		\|u_{k}\|_{L^1 (\Omega)} & \le \|f \delta^s \|_{L^1 (\Omega)}\\
		\|V_{k} u_{k} \delta^s \|_{L^1 (\Omega)} & \le \|f \delta^s \|_{L^1 (\Omega)}.
		\end{align}
		
It is easy to prove that for $k_1 < k_2$ we have
		\begin{align}
		0 \le u_{k_2}& \le u_{k_1}.
		\end{align}
Hence, by the Monotone Convergence Theorem we know that there exists $u \in L^1 (\Omega)$ such that
		\begin{equation}
		u_{k} \to u \qquad \textrm{a.e. and in } L^1 (\Omega).
		\end{equation}	
Furthermore,
		\begin{equation}
		\| u_{k } \|_{L^\infty} \le  c \| f \|_{L^\infty}.
		\end{equation}
Hence,
		\begin{equation}
		u_{k} \rightharpoonup u \qquad L^\infty \textrm{-weak-} \star.
		\end{equation}
Let $K \Subset \Omega$ be a compact set. We have that
		\begin{equation} \label{eq:proof theorem existence need V in L1 loc}
		\|V_k u_{k} \|_{L^1 (K) } \le c\|V\|_{L^1 (K)} \|f\|_{L^\infty}.
		\end{equation}
Notice that this is not true if $K$ is replaced by $\Omega$. Also
		\begin{equation}
		0 \le V_k u_{k} \delta^s  \le V u_{k} \delta^s \le V u_{0} \delta^s .
		\end{equation}
By the Dominated Convergence Theorem
		\begin{equation}\label{conv.local}
		V_k u_{k} \delta^s  \to V u \delta^s \qquad L^1(K).
		\end{equation}
We have proved, therefore, that
		\begin{equation}
		\int_ \Omega u (-\Delta)^s \varphi + \int _ \Omega V u \varphi = \int_ \Omega f \varphi, \qquad \forall \varphi \in C_c^\infty (\Omega).
		\end{equation}
		
		This completes the proof of existence of a solution $u$ of \eqref{eq:FDE local} for $f \ge 0$ and bounded.

(2) We improve the result, still keeping   $f$ bounded. Since $0 \le f \varphi_{ \delta } \in L^1 (\Omega)$, then \eqref{eq:u deltas in L1} and \eqref{eq:V u phi delta in L1} hold for $u_{k}$ and $V_k u_{k}$. It is easy to check, applying Fatou's lemma, that the estimates hold for $u$ and $Vu$. In particular, $ u /{\delta^s} \in L^1(\Omega)$ and $Vu\delta^s \in L^1 (\Omega)$. Applying \Cref{thm:equivalence of formulations} we deduce that it is a solution of \eqref{eq:FDE Brezis}. Hence,
			\begin{equation}
			\int_ \Omega u (-\Delta)^s \varphi + \int _ \Omega V u \varphi = \int_ \Omega f \varphi, \qquad \forall \varphi \in \Xs.
\end{equation}
		
(3) Assume now that $0 \le f \in L^1 (\Omega, \delta^s)$. Let $f_m = \min \{ f, m \}$ and let $u_m$ be the solution of
\begin{equation}
		\begin{dcases}
		(-\Delta)^s u_{m} + V u_{m} = f_m & \Omega, \\
		u = 0 & \Omega^c.
		\end{dcases}
		\end{equation}
Since $f_m$ is a pointwise nondecreasing sequence, then $u_m$, $Vu_m$ and $Vu_m \delta^s $ are pointwise nondecreasing sequences. If  $m_1 < m_2$ then
		\begin{align}
		u_{m_1}& \le u_{m_2}.
		\end{align}
The sequence of functions $u_m$ converges in $L^1 (\Omega)$ due to the Monotone Convergence Theorem since it is uniformly bounded above in $L^1$. We have
		\begin{equation}
		\| u \|_{L^1 (\Omega)} = \lim_{m \to \infty } \|u_m\|_{L^1 (\Omega)} \le c  \lim_{m \to \infty } \|f_m   \delta^s \|_{L^1 (\Omega)} =  \|f   \delta^s \|_{L^1 (\Omega)}.
		\end{equation}
 Likewise
		\begin{equation}
		\|V u \delta^s \|_{L^1 (\Omega)} = \lim_{m \to \infty } \|V u_m \delta^s \|_{L^1 (\Omega)} \le c  \lim_{m \to \infty } \|f_m   \delta^s \|_{L^1 (\Omega)} =  \|f   \delta^s \|_{L^1 (\Omega)}.
		\end{equation}
and $V u_m\to V u$ in $L^1(\Omega; \delta^s)$  by monotone convergence.
We can now pass  to the limit in the very weak formulations to show that
		\begin{equation}
		\int_ \Omega u (-\Delta)^s \varphi + \int _ \Omega V u \varphi = \int_ \Omega f \varphi, \qquad \forall \varphi \in C_c^\infty (\Omega).
		\end{equation}

		This proves existence of a very weak solution when $f \ge 0$ and also positivity (ii).

(4)		To prove item (iii) when $f \ge 0$, we assume that $0 \le f \varphi_{ \delta } \in L^1 (\Omega)$. Then \eqref{eq:u deltas in L1} and \eqref{eq:V u phi delta in L1} hold for the sequences $u_{k}, u_m$ and $V_k u_{k}, Vu_{m}$ that appear in steps (1)--(3) of the previous proof. It is easy to check that, in each of the limits, the estimates hold.

	(5) In all the limits, applying \eqref{eq:u in L2} and \eqref{eq:estimate Ls u} we know that $\| u_m \|_{H^s(\mathbb R^n)}, \| u_k \|_{H^s(\mathbb R^n)} \le C \| f \|_{L^2}$, and so it converges weakly in $H^s (\mathbb R^n)$. In particular, the limit $u \in H^s_0 (\Omega)$ and \eqref{eq:u in L2} and \eqref{eq:estimate Ls u} hold.

	(6) In order to prove items (i), (ii) and (iii) when $f$ changes sign, we can split $f = f_1 - f_2$ where $f_i \ge 0$. We apply the previous part of the proof for $f_i$ to construct $u_1$ and $u_2$. We define $u = u_1 - u_2$. This concludes the proof.
\end{proof}

\begin{remark}
		An analogous way to complete step (2) is to realize that $V_k u_k \varphi_ \delta \in L^1 (\Omega)$ with uniform bounds. By splitting the integrals near and far from the boundary, and using the sharp estimates for $\varphi_\delta$, we can check that $V_k u_k \delta^s$ converges in $L^1 (\Omega)$.
	\end{remark}
	
This result can be extended to measures as data, in the space ${\mathcal M}(\Omega, \delta^s)$ by taking limits. See comments on  \Cref{sec.comment}.


\subsection{Accretivity and counterexample}

The  results of the preceding subsections allow to prove the following extension of the results for the operator without potential.

\begin{corollary} \label{cor.accret} The fractional  operator $L_V=(-\Delta)^s+ V$ with $V\ge 0$, $V\in L^1_{loc}(\Omega)$ is $m$-$T$-accretive in the space $L^1(\Omega)$ and also in the spaces $L^1(\Omega; \phi)$ for all positive weights  $\phi\in X^s$, such that $(-\Delta)^s \phi\ge 0$. Moreover, $L_V$ is accretive in all the spaces $L^p(\Omega)$, $1\le p \le \infty$.
\end{corollary}\normalcolor

As a negative result for operators with potentials, we want to show that for unbounded potentials $V\ge 0$ the requirement that $f\in L^\infty(\Omega) $ does not imply, in general, that $Vu\in L^\infty(\Omega) $, where $u$ is the solution of $(-\Delta)^su+ Vu=f$.

\noindent {\bf Construction of a counterexample. } (i) We consider a nice bounded, positive and smooth function $f\ge0$ defined in $\Omega$, we may also assume that $f$ has compact support in $\Omega$; we also take a nice bounded potential $V_1\ge0$, and consider the solutions of the Dirichlet problem in $\Omega$
\begin{equation}
(-\Delta)^s u_0 =f,\quad (-\Delta)^s u_1+V_1(x)u_1 =f.
\end{equation}
By the theory of preceding sections, both solutions are bounded and nonnegative in $\Omega$ and
$0\le u_1\le u_0$. Since $V_1u_1$ is bounded the theory says that the solution $u_1$ is also $C^s$ in $\Omega$.

Take any point $x_0\in\Omega$ where $u_1$ is strictly positive $u_1(x_0)=c_0>0$.
By continuity we can take a small ball $B_{r_0}(x_0)\subset \Omega$ where $u_1(x)=c_0/2>0$.

(ii) We now take a perturbation $g(x)=G(|x-x_0|)\ge 0$ which is radially symmetric and decreasing around $x_0$ and is supported maybe in $B_{r_0}(x_0)$. Consider now the solution $u_2\ge 0$ of the Dirichlet problem in $\Omega$
\begin{equation}
(-\Delta)^s u_2+V_2(x)u_2 =f\, \quad V_2(x)=V_1(x)+ g(x).
\end{equation}
We have $0\le u_2(x)\le u_1(x)$ in $\Omega$.

(iii) Let us prove that $u_2(x)$ is uniformly positive if $g\in L^p(\Omega)$ with a small bound. In fact if $u=u_1-u_2$ we have
\begin{equation}
(-\Delta)^s u + V_1(x)u = g(x)u_2\le Cg(x).
\end{equation}
By the known embedding theorems or using the bounds for the Green function, we conclude that when $p$ is large enough, $p>p(s,n)$ we have $u\in C(\Omega)$ and
$$
0\le u(x)\le c_1(p,s,n)\|g\|_p\,.
$$
Therefore, if $\|g\|_p$ is small enough we have $ u(x)\le c_0/4$. Note that $c_0$ was defined before and does not depend on the perturbation $g$. It follows that
$$
u_2(x)=u_1(x)-u(x)\ge c_0/4 \qquad \mbox{in } B_r(x_0).
$$

(iv) We now impose the last requirement, $g(x_0)=+\infty$. Then we have
$$
V_2(x_0)u_2(x_0)=+\infty.
$$
Moreover, we can easily find functions $g\not\in L^q(B_r(x_0))$ with $q>p>p(s,n)$. Therefore, in that case
$$
\| V_2u_2\|_q=+\infty.
$$

\begin{remark} The construction can be generalized to cases with blow-up of $Vu$ at many points; and maybe the requirement $q>p(n,s)$ can be eliminated, so that bound of $Vu$ by means of $f$ in $L^p$ is false for $p>1$.
\end{remark}

\subsection{Solutions with measure data}

As we pointed out in the introduction, there is a simple extension of our existence and uniqueness theory as reflected in Theorems \ref{thm:uniqueness} and \ref{thm:existence} to right-hand side of the equation is a measure $\mu\in {\mathcal M}(\Omega, \delta^s)$. We leave it to the reader to prove that both mentioned theorems
hold in that generality. \normalcolor

	For \eqref{eq:FDE Laplace}, a nice theory can be done, as well, through the Green kernel. Many of the results and techniques in \cite{Marcus+Veron:2013} still hold in this setting.

\section{Super-singular potentials}\label{sec.sspot}

In this section we discuss the influence on the theory of potentials $V\in L^1(\Omega)$ that blow up near the boundary. We are in particular interested in potentials $V \ge C /\delta^{2s}$ that we call {\sl super-singular potentials}. This kind of potentials is very relevant in Physics (see, e.g., \cite{Cooper1995,Pochl1933}). Surprisingly,  potentials with large blow-up on $\partial \Omega$ are very good for the  theory we have described above, as we will show next.

\subsection{Definitions and first results}

We want to address  the question of how super-singular potentials regularize the solutions. 	A  problem with the regularity of solutions of problems involving the fractional Laplacian operator $(-\Delta)^{s}$ is that,  according to \Cref{prop:Ros-Oton} and \Cref{prop:Hopf}, the solutions are typically $u \asymp \delta^s$, which is not of class $\mathcal C^1$ in a neighbourhood in $\Omega$ of $\partial \Omega$. However, for super-singular potentials, solutions such that $u / \delta^{s+\varepsilon} \in L^\infty (\Omega)$ may be found (see \Cref{thm:flatness barrier}.  Hence, we have a higher H\"older exponent at $\partial \Omega$.   In this sense, super-singular potentials force the solution $u$ to be more regular in the proximity of the boundary. A natural definition of the concept of flat solution for $s<1$ is the following
	
\begin{definition}
		We say that $u \in L^1 (\Omega)$ is an $s$-flat solution if, for every $y \in \partial \Omega$
		\begin{equation}
			\lim_{x \to y} \frac{u(x)}{\delta(x)^s} = 0.
		\end{equation}
	\end{definition}
Clearly, a sufficient condition for $s$-flatness is that $u / \delta(x)^{s+\ee} \in L^\infty (\Omega)$ for some $\ee > 0$.

We first obtain a result on existence of flat solutions in a weaker integral sense  that follows directly from our results in previous sections:

\begin{proposition}\label{rem:V super singular u delta s}
		If $V \ge C /\delta^{2s+\varepsilon}$ for some $C>0$, $\varepsilon \ge 0$, then for every $f\in L^1(\Omega; \delta^s)$ we have ${u}/{\delta^{s+\varepsilon}} \in L^1 (\Omega)$, even if  $f \varphi_{ \delta } \notin L^1 (\Omega)$.
\end{proposition}

Indeed, we have proved that $Vu \delta^s \in L^1(\Omega)$, so that the lower bound for $V$ implies the conclusion.
		
\begin{remarks} 1) When $V \ge c \delta^{-2s}$ , the equivalence \eqref{eq:equivalence u deltas if V is 0} no longer holds. The sufficiency part still holds, but this result here shows that  the extra condition on $f$ is no longer necessary.

2) 	The integral sense is not a very strong concept of flat solution, but it is nevertheless credited in the literature for $s=1$.
		In that case ($s=1$) super-singular potentials have been shown to ``flatten'' the solution, in the sense that ${\partial u}/{\partial n} = 0$ on $\partial \Omega$. For large powers, higher order derivatives vanish. For further reference, see \cite{Diaz2015ambiguousSchrodinger,Diaz2017ambiguousSchrodinger,diaz2019ambiguous,Orsina2017}.
	\end{remarks}

\subsection{Pointwise flatness estimates of solutions through barrier functions}

We give next conditions for $s$-flatness in the everywhere sense. In order to prove this fact we will construct some clever barrier function (as in \cite{diaz2019ambiguous}). We first need a technical result.

	\begin{lemma}
	Let $\nu_ \beta (x) =  |x|^{\beta }$ with $\beta > 0$. Then
	\begin{equation}
		(-\Delta)^s \nu_\beta = \gamma_ \beta |x|^{-2s} \nu_\beta, \qquad \textrm{ in } \mathbb R^n,
	\end{equation}
	where
	\begin{equation}\label{formula.gamma}
		\gamma_\beta =%
		 2^{2s}\frac{\Gamma\left(\frac{n+\beta}{2}\right) \Gamma\left( s- \frac \beta 2 \right) }%
		{\Gamma\left(-\frac \beta 2 \right) \Gamma\left ( - s + \frac{\beta  + n}{2}\right)}
	\end{equation}
	is a constant.
\end{lemma}
The computation of the result above can be found in \cite[p.798]{Vazquez2014}
	and \cite{Fall2012}. It can be obtained by applying the Fourier transform formula of a radial function given in \cite[Theorem 4.1]{stein2016fourier}.
Note that  $\gamma_{s+\ee} <0 $ for $0<\varepsilon  < s$, while $\gamma_{2s}$ diverges.

\begin{lemma}
	Let $0 < \varepsilon  < s$, $0 \le f \in L^\infty$, $V \ge C_V |x-x_0|^{-2s}$ with $ C_V>-\gamma_{s+\varepsilon} >0$,   and let $x_0 \in \partial \Omega$. Then,
	\begin{equation}
		\frac{u(x)}{|x-x_0|^{s+\varepsilon}} \le \frac{\|f \|_{L^\infty}}{(\gamma_{s+\varepsilon} + C_V) } R(x_0)^{s-\varepsilon},
	\end{equation}
	a.e. in $\Omega$, where $R(x_0) = \max_{x \in \bar \Omega} d(x,x_0)$ (i.e. such that $\Omega \subset B_{R}(x_0)$).
\end{lemma}

\begin{proof}
	Since $0 \le f \in L^\infty$ we have that $0\le u \in L^\infty$.
	
	Let us consider $U (x) = C_U \nu_{s + \varepsilon}(x-x_0)$ where
	\begin{equation}
		C_U = \frac{\|f \|_{L^\infty}}{(\gamma_{s+\varepsilon} + C_V) }R^{s-\varepsilon} > 0.
	\end{equation}	
	We compute
	\begin{align}
		(-\Delta)^s U + V U &= \gamma_{s+\varepsilon} |x-x_0|^{-2s} U + V U\\
		& \ge (\gamma_{s+\varepsilon} + C_V) |x-x_0|^{-2s} U\\
		&= C_U (\gamma_{s+\varepsilon} + C_V) |x-x_0|^{-s+\varepsilon}  \\
		&\ge C_U (\gamma_{s+\varepsilon} + C_V) R^{-s+\varepsilon} \\
		&= \| f \|_{L^\infty}
	\end{align}
	a.e. in $\Omega$, since $-s + \varepsilon < 0$.
	Since also $U \ge 0 = u$ on $\Omega^c$ we have that $U \ge u$ a.e. in $\Omega$. Therefore,
	\begin{equation}
		\frac{u}{|x-x_0|^{s+\varepsilon}} \le \frac{U}{|x-x_0|^{s+\varepsilon}} = C_U.
	\end{equation}
	a.e. in $\Omega$. This completes the proof.
\end{proof}

\begin{theorem} \label{thm:flatness barrier}
	Let $0 < \varepsilon  < s$, $0 \le f \in L^\infty$, $V (x) \ge C_V \delta (x)^{-2s}\ge 0$ with $C_V > - \gamma_{s+\varepsilon}$. Then,
	\begin{equation}
		\frac{u}{\delta^{s+\varepsilon} } \in L^\infty (\Omega).
	\end{equation}
\end{theorem}

\begin{proof}
	Since $\delta(x) = \min_{x_0 \in \partial \Omega} |x-x_0|$ we have that
	\begin{equation}
		\frac{u(x)}{\delta(x)^{s+\varepsilon}} = \max_{x_0 \in \partial \Omega} \frac{u(x)}{|x-x_0|^{s+\varepsilon}} \le \frac{\|f \|_{L^\infty}}{(\gamma_{s+\varepsilon} + C_V) } \max_{x_0 \in \bar \Omega} R(x_0)^{s-\varepsilon}.
	\end{equation}
	This last maximum if finite because $\Omega$ is a bounded set. This completes the proof.
\end{proof}

\begin{remarks}
	1) We have that
	\begin{equation}
		\frac{u(x)}{\delta^s(x)} \le C \delta^\varepsilon(x) \to 0
	\end{equation}
	uniformly as $x \to \partial \Omega$. These functions are \emph{uniformly} $s$-flat.
	
	Notice that this implies the unique continuation property (see, e.g., \cite{Fall2015}) fails for super-singular negative potentials.
	
	2)		Notice that the P\"osch-Teller potential \eqref{Posch-Teller} is $L^1_{loc} (\Omega)$ and behaves like $V \ge c d (x, \partial \Omega)^{-2}$ in any annulus of the form
		\begin{equation}
		\Omega_k = \left \{ x \in \mathbb R^n :  {k\pi} < \alpha |x| \le {k \pi + \frac \pi 2} \right\}.
		\end{equation}	
		or
		\begin{equation}
		\Omega_k = \left \{ x \in \mathbb R^n :  {k\pi} + \frac \pi 2 < \alpha |x| \le {(k+1) \pi} \right\}.
		\end{equation}
	
	3)	The results presented here are part of an ongoing research and  must be improved.
\end{remarks}
\normalcolor

\subsection{Pointwise flatness estimates for radially symmetric data}

There is still the question of getting solutions as flat as desired if $V$ is singular enough near the border. We do not have a general result in that direction, since it needs more tools and space, but we do have a convincing example. It is as follows:

\begin{theorem}\label{thm:point.flat} Let $\Omega=B_R(0)$ be a ball, $f\in L^1({\Omega})$ positive, radially symmetric and decreasing, and  let the potential $V$ be positive, radially symmetric and increasing. \normalcolor Then the solution is nonnegative, radially symmetric and non-increasing. If moreover $V(x)\ge C_V \delta(x)^{-p }$ for some $C_V>0$ and $p>1$, then we have
\begin{equation}\label{form.point.est}
u(x)\le \frac{C_1}{C_V}\|f\|_1\delta(x)^{p-1},
\end{equation}
when $\frac R 2 < |x| \le R$, where $C_1>0$ depends on $n,s$. The same conclusion holds if $f\ge 0 $ is not radially symmetric but $f\le g$ where $g$ is positive, integrable, radially symmetric and decreasing. Then formula {\rm (\ref{form.point.est})} holds if we replace $\|f\|_1$ with $\|g\|_1$.
\end{theorem}

\begin{proof} (i) We assume first  that $f$ is positive, radial and decreasing, hence $f=g$. In that case we already know that $u \ge 0$. The fact that  $u$ is radially symmetric follows from uniqueness and the invariance of the problem under rotations. The fact that $u(r)$ is nonincreasing can be proved by a modification of the Aleksandrov reflection principle proved in \cite{Vazquez2014}, Section 15.
	
To get the estimate we first integrate in $\Omega$ we get
$$
\int_{\Omega}(-\Delta)^s u(x) \,dx  + \int_{\Omega}V(x)u(x)\,dx= \int_{\Omega}f(x)\,dx
$$
It is easy to prove that for $u\ge 0$ there is the inequality $\int_{\Omega}(-\Delta)^s u(x) \,dx\ge 0$. We recall that such an inequality is quite standard in the classical Laplacian case. In order to prove it we may use the formula for the operator and get
$$
\int_{\Omega}(-\Delta)^s u(x) \,dx =\int_{x\in \Omega}\int_{y\not\in \Omega}
\frac{u(x)}{|x-y|^{n+2s}}\,dxdy\ge 0,
$$
since the remaining integral cancels by symmetry:
$$
\int_{x\in \Omega}\int_{y\in \Omega}\frac{u(x)-u(y)}{|x-y|^{n+2s}}\,dxdy= 0.
$$
 We may use all this to conclude that
\begin{equation} \label{int.1}
	C_V\int_B \frac{u(x)}{(R-|x|)^p}\le \int_{B}V(x)u(x) dx\le \|f\|_1.
\end{equation}
%
Note that, in this case, $\delta(x)=R-|x|$ for $x\in \Omega=B_R(0)$. Take now a fixed point $x_0\in {\Omega}$ {such that $\frac R 2 < |x_0| \le R$ (i.e.,} near the border) and consider the annulus $A$ with outer radius $r_2=|x_0|$ and inner radius $r_1=r_2-\delta({x_0}) {=2|x_0|-R}>0$. Then, by monotonicity and radial symmetry
\begin{equation} \label{eq:flatness example 2}
	\int_{A}\frac{u(x)}{(R-|x|)^p}\,dx\ge {\frac{u(x_0)}{(R-r_1)^p}|A|=}  \frac{u(x_0)}{(2\delta(x_0))^p}|A|= c u(x_0)\delta {(x_0)}^{1-p}\,.
\end{equation}
Since $A\subset {\Omega}$, combining \eqref{int.1} and \eqref{eq:flatness example 2} we have
\begin{equation}
	u(x_0)\le \frac{C_1}{C_V}(R-|x_0|)^{p-1}\|f\|_1\,.
\end{equation}
This completes the proof under the stated assumptions.

\noindent (ii) In the case where $f$ is not necessarily radially symmetric but it is bounded above by $g$, we can solve the problem with right-hand side $g$, obtain a solution $v$  that satisfies the desired estimate and then we can use the comparison theorem, $u\le v$. \end{proof}

We may also take a nonradial $V$ that has a supersingular radial lower bound $V_1$ like in the Theorem, and apply again the maximum principle to conclude the same type of bound;  we leave the easy detail to the reader. We also note that the argument can be extended to other equations and more general domains, but since other techniques are needed this extension will not be discussed here.

Though this example does not give optimal rates,  it shows the existence of solutions that are as flat as we like near the boundary, always depending on the divergence of the potential. In that sense,  we have the following interesting consequence.

\begin{corollary} Under the assumptions of \Cref{thm:point.flat}, if moreover  $V$ satisfies
$$
\lim_{\delta(x)\to 0} V(x)\delta(x)^p=+\infty \quad \mbox{for every $p>1$,}
$$
then the solution $u(x)$ vanishes at the boundary to infinite order in the sense that
 \begin{equation}
 \lim_{\delta(x)\to 0}\frac{u(x)}{\delta(x)^p}=0 \quad  \mbox{for every $p>1$.}
 \end{equation}

\end{corollary}


\section{The restricted fractional Laplacian as the natural limit of the Schr\"odinger equation in $\mathbb R^n$ for the super-singular potential} \label{sec:natural limit}

Let us consider the singular infinite well potential (see the exposition in \cite{Diaz2015ambiguousSchrodinger,Diaz2017ambiguousSchrodinger} for $s=1$ and \cite{diaz2019ambiguous} for $0<s<1$)
\begin{equation} \label{eq:infinite well}
	V (x) = \begin{dcases}
		d(x, \partial \Omega)^{-2s}& \Omega, \\
		+\infty & \Omega^c .
	\end{dcases}
	\end{equation}
To avoid the ambiguity of the definition of $Vu$ in $\Omega^c$,  the solutions of the associated Schr\"odinger problem can be understood as the limit of the solutions of the corresponding finite-well potentials
\begin{equation}
	V_k(x) = k \wedge V(x).
\end{equation}
The stationary Schr\"odinger equation over its natural domain, the whole space, corresponds to finding $u_k \in H^s (\mathbb R^n)$ such that
\begin{equation}
	\begin{dcases}
	(-\Delta)^{s} u_k + V_k(x) u_k = f & \mathbb R^n, \\
	u_k \to 0 & |x| \to +\infty,
	\end{dcases}
\end{equation}
for some function $0 \le f \in L^\infty (\Omega)$, $f = 0$ in $\Omega^c$. Here, all the usual formulations are equivalent. Hence $u_k \ge 0$. Furthermore, $0 \le u_k$ is a decreasing sequence, and hence has limit in $L^1 (\Omega)$, $0 \le u \in L^1 (\mathbb R^n)$ which is also an a.e. pointwise limit, due to the Monotone Convergence Theorem.

\begin{theorem}
	Assume \eqref{eq:infinite well}. As $k\to \infty$ the solutions of the approximate problems in $\mathbb R^n$ converge to the solution of Problem \eqref{eq:FDE}. In particular $u = 0$ in $\Omega^c$.
\end{theorem}

\begin{proof}
	Using the solution of
	\begin{equation}
		\begin{dcases}
			(-\Delta)^s \varphi_0 = 1 & \mathbb R^n , \\
			\varphi \to 0 & |x| \to \infty
		\end{dcases}
	\end{equation}
	we deduce that
	\begin{equation}
		(1 + k \min_{\Omega_c} \varphi_0) \int_{ \Omega^c } u_k \le \int_ {\Omega} f \varphi_0
	\end{equation}
	Hence $u = 0$ in $\Omega^c$.
	
	On the other hand,
	\begin{equation}
		\int_{\mathbb R^n} u_k (-\Delta)^s \varphi + \int_ {\mathbb R^n} V_k u_k \varphi = \int_{\Omega} f \varphi.
	\end{equation}
	As before, for $K \subset \Omega$ compact $V_k u_k \to Vu$ in $L^1 (K)$ by the Dominated Convergence Theorem.\\
	Finally, for any $\varphi \in \mathcal C_c^\infty(\Omega)$ such that $(-\Delta)^s \varphi \in L^\infty (\mathbb R^n)$, we pass to the limit to obtain
	\begin{equation}
		\int_{\mathbb R^n} u (-\Delta)^s \varphi + \int_ {\mathbb R^n} V u \varphi = \int_{\Omega} f \varphi.
	\end{equation}
	For $\varphi$ the restricted fractional Laplacian and the fractional Laplacian in  $\mathbb R^n$ coincide.
	
	Since $u = 0$ in $\Omega_c$ and $\varphi$ is supported in $\Omega$, this is precisely
	\begin{equation}
	\int_{\Omega} u (-\Delta)^s \varphi + \int_ {\Omega} V u \varphi = \int_{\Omega} f \varphi.
	\end{equation}
	By density, we have the previous formulation for all $\varphi \in \Xs \cap C_c (\Omega)$.
\end{proof}

This shows that the natural fractional Laplacian to deal with the Schr\"odinger equation with the singular infinite-well potential problem is the restricted fractional Laplacian over $\Omega$. We point out that, physically, the Schr\"odinger equation a priori must be defined over the whole space, $\mathbb R^n$, and that any other constraint (as, for instance, to assume a localization to a subset $\Omega$) must be justified.

\section{Another perspective on the results}\label{sec.persp}

\subsection{The Green operator's viewpoint} \label{sec:Green operator viewpoint}


Let us start for the case $V = 0$. As mentioned, for regular $f$ we know that the unique solution of
\begin{equation}
\begin{dcases}
	(-\Delta)^s  u = f & \Omega, \\
	u = 0 & \Omega^c,
\end{dcases}
\end{equation}
is written in the form
\begin{equation}
	u(x) = \int_ {\Omega} \mathbb G_s (x,y) f (y) dy
\end{equation}
where $\mathbb G_s$ satisfies \eqref{eq:Green function estimates}. In \Cref{sec:L1 weighted optimal class} we have shown that the optimal set of data functions for the Green kernel is given by:
\begin{subequations}
\begin{eqnarray}
	\Gs: \Dom(\Gs) = L^1 (\Omega, \delta^s) &\longrightarrow&  L^1_0 (\Omega) \\
	f & \longmapsto&  \chi_ \Omega (\cdot) \int_{ \Omega } \mathbb G_s (\cdot, y) f(y) dy,
\end{eqnarray}
\end{subequations}
where $\chi_\Omega$ is the characteristic function of $\Omega$,  and
\begin{equation}
	L_0^1 (\Omega) = \{  u \in L^1 (\mathbb R^n): u = 0 \textrm{ in } \Omega^c \}.
\end{equation}

In this sense we can characterize
\begin{equation}
	\Xs = \Gs (L^\infty (\Omega)).
\end{equation}
For this, let us read the very weak formulation of \eqref{eq:FDE Laplace} in terms of $\Gs$. The very weak formulation \eqref{eq:FDE Brezis} reduces to
\begin{equation} \label{eq:vwf V 0}
	\int_{ \Omega } u ( -\Delta) ^s \varphi = \int_ {\Omega} f \varphi , \qquad \forall \varphi \in \Xs.
\end{equation}
First, for $f \in L^\infty (\Omega)$, we point out that, as the unique solution is $u = \Gs (f)$, we can write
\begin{equation}
\int_{ \Omega } \Gs(f) ( -\Delta) ^s \varphi = \int_ {\Omega} f \varphi , \qquad \forall \varphi \in \Xs.
\end{equation}
Since,  $\Xs = \Gs(L^\infty)$, we can write $\varphi = \Gs (\psi)$ for some $\psi \in L^\infty (\Omega)$, and so $(-\Delta)^s \varphi = \psi$. Therefore, \eqref{eq:vwf V 0} is equivalent to
\begin{equation} \label{eq:Gs self-adjoint}
	\int_{ \Omega } \Gs(f) \psi = \int_ {\Omega} f \Gs (\psi), \qquad \forall \psi \in L^\infty(\Omega).
\end{equation}
Thus, the very weak formulation for $f \in L^\infty$ is equivalent the fact that $\Gs$ is self-adjoint.

The following result gives a direct answer:
\begin{proposition} \label{eq:kernel solution is vws}
	Let $L^\infty (\Omega) \subset Y \subset \Dom (\Gs)$ be such that
	\begin{equation}
	\Gs: Y \to L^1 ( \Omega) \textrm{ is continuous}
	\end{equation}
	and assume that $L^\infty(\Omega)$ is dense $Y$. Then $\Gs(f)$ is a very weak solution of \eqref{eq:FDE Laplace}.
\end{proposition}

\begin{proof}
	Let $f \in Y$ and $f_k \in L^\infty( \Omega)$ be a sequence converging to $f$ in $Y$. Then $\Gs(f_k)\to \Gs(f)$ in $L^1 (\Omega)$. On the other hand $\Gs(f_k)$ is a very weak solution of  \eqref{eq:FDE Laplace}. By passing to the limit in \eqref{eq:Gs self-adjoint}, we deduce that $\Gs(f)$ also satisfies \eqref{eq:Gs self-adjoint}, and so it is a very weak solution.
\end{proof}

It was shown in \cite{Ros-Oton2014} $\Gs : L^\infty( \Omega) \to \mathcal C^s (\Omega)$ is continuous. In \cite{chen+veron2014} the authors showed that $\Gs: L^1 (\Omega, \delta^s) \to L^1 (\Omega)$ is also continuous. We have shown here that $\Gs : L^1 (\Omega, \varphi_{\delta}) \to L^1 (\Omega, \delta^{-s})$ is also continuous. Furthermore,
\begin{equation}
	\Gs^{-1} \Big(  L^1 (\Omega, \delta^{-s}) \cap \mathrm{Im}(\Gs) \Big) = L^1 (\Omega, \varphi_{\delta}).
\end{equation}

Problem \eqref{eq:FDE} with $V \ne 0$ is also linear, and could allow for another Green kernel. However, we can write \eqref{eq:FDE} as a fixed point problem for the Green operator of $(-\Delta)^s$ as:
\begin{equation}
	u = \Gs ( f - Vu) .
\end{equation}
Let $u$ be a solution, and let $g = f - Vu$. In \Cref{lem:estimates V and f bounded} we show that, if $V, f \in L^\infty$, then
\begin{subequations}
\begin{align}
	\| g \delta^s \|_{L^1}&\le C \| f \delta^s \|_{L^1} \\
	\| g \varphi_{ \delta } \|_{L^1}&\le 2 \| f \varphi_{ \delta } \|_{L^1}.
\end{align}
\end{subequations}
The results in Section 4 of this paper lead to corresponding properties of the Green operator for \eqref{eq:FDE}
\begin{equation}
	\GsV: f \mapsto u.
\end{equation}
So far, we have proved that:
\begin{enumerate}
	\item If $V \in L^1_{loc}$ then {$\GsV : L^1 (\Omega, \delta^s) \to L^1 (\Omega)$ and} $\GsV: L^1 (\Omega, \varphi_{\delta}) \to L^1 (\Omega, \delta^{-s})$.
	\item If $V \ge \delta^{-2s}$ then
	$\GsV: L^1 (\Omega, \delta^s) \to L^1 (\Omega, \delta^{-s}).$
\end{enumerate}
It is easy to show that
\begin{equation}
	\GsV (x,y) \le \Gs (x,y) \qquad \forall x, y \in \Omega.
\end{equation}
For $V \in L^\infty$ it is likely that
\begin{equation}
	\GsV (x,y) \asymp \Gs (x,y).
\end{equation}
However, the additional integrability for the case $V \ge c \delta^{-2s}$ guaranties that
\begin{equation}
\GsV (x,y) \not \asymp \Gs (x,y).
\end{equation}

\subsection{What is $(-\Delta)^s$ of a very weak solutions?}
Let us think about $(-\Delta)^s$ as a functional operator. It is natural to define
\begin{subequations}
\begin{eqnarray}
	\Ls :\Dom(\Ls) \subset L_0^0 (\Omega) & \longrightarrow & L^0 (\Omega) \\
	u & \longmapsto & c_{n,s} P.V. \int_{\mathbb R^n} \frac{u(\cdot) - u(y)}{|\cdot-y|^{n+2s}} dy, \label{eq:Ls mapsto}
\end{eqnarray}
\end{subequations}
where $L^0 (\Omega)$ is the set of measurable functions in $\Omega$ and
\begin{equation}
	L_0^0 (\Omega) = \{ u \in L^0 (\mathbb R^n) : u = 0 \textrm{ in } \Omega^c  \}.
\end{equation}
It is easy to show that
\begin{equation}
	\Ls : \mathcal C^{2s}_0 (\overline \Omega)  \longrightarrow  \mathcal C^s( \bar \Omega) \\
\end{equation}
where
\begin{equation}
	\mathcal C^{2s}_0 (\overline \Omega) = \{  u \in \mathcal C^{2s} (\mathbb R^n): u = 0 \textrm{ in } \Omega^c  \}.
\end{equation}
However, working with integrable rather than smooth functions $f$, we do not expect $u \in \mathcal C^{2s}_0 (\overline \Omega)$. Nonetheless, our aim is to solve the problem \eqref{eq:FDE}, so we are interested in the definition of $(-\Delta)^s u$.

By the regularization results obtained through H\"ormander theory in \cite{Grubb2015,Ros-Oton+Serra2016Duke}, we have that
\begin{equation}
	\Gs: \mathcal C^{\gamma} (\bar \Omega) \to \mathcal C^{\gamma + s} (\bar \Omega, \delta^{-s}),
\end{equation}
if $\gamma + s \notin \mathbb N$. We point that $\mathcal C^{\gamma + s} (\bar \Omega, \delta^{-s}) \subset \mathcal C^{\gamma + s} (\Omega) \cap \mathcal C_0 (\bar \Omega)$.
By uniqueness of solutions \eqref{eq:FDE Laplace} it is clear that
\begin{align}
	u & = \Gs \Ls u \qquad u \in \mathcal C^{2s} (\bar \Omega, \delta^{-s}),\\
	f &= \Ls \Gs f \qquad f \in \mathcal C^{s} (\bar \Omega). \label{eq:f to f Cs}
\end{align}
We can extend this result to an abstract setting.
In this direction we have:
\begin{proposition}
	Let $X \subset \Dom (\Ls)$ and $\mathcal C^{s} (\bar \Omega) \subset Y \subset \Dom (\Gs)$. Assume that $\Gs:Y \to X$ and $\Ls : \Gs(Y)  \to Y$ are continuous and that $\mathcal C^{s} (\bar \Omega)$ is dense in $Y$. Then
	\begin{equation} \label{eq:f to f}
		f = \Ls \Gs f \textrm{ in } Y.
	\end{equation}
\end{proposition}
\begin{proof}
	Let $f_k \in \mathcal C^{s} (\bar \Omega)$ be a sequence such that $f_k \to f$ in $Y$. Then $\Gs f_k \to \Gs f$ in $X$. On the other hand, from \eqref{eq:f to f Cs} we know that
	$
		f_k = \Ls \Gs f_k.
	$
	Therefore $f = \Ls \Gs$.
\end{proof}

If we get inspiration in the case of usual Laplacian we soon see that this pointwise construction, although natural, is not optimal working grounds. By looking again at the case of the usual Laplacian, we would like to study a distributional formulation. By \Cref{prop:fractional integration by parts} we have that $\Ls|_{\mathcal C^{2s} (\bar \Omega)}$ is self-adjoint. We can define a self-adjoint extension as a distributional operator
\begin{eqnarray}
	\widetilde \Ls : L^1 (\Omega) \to \mathcal D' (\Omega)
\end{eqnarray}
through the notion of very weak solution, i.e.
\begin{subequations}
	\begin{eqnarray}
	\widetilde \Ls u : \mathcal C_c^\infty (\Omega) & \longrightarrow& \mathbb R \\
	\varphi  &\mapsto& \int_{ \Omega } u (\Ls \varphi).
	\end{eqnarray}
\end{subequations}
Through \Cref{prop:fractional integration by parts} we know that, for $u \in \mathcal C^{2s}_0 (\overline \Omega)$,
\begin{equation}
	\langle \widetilde \Ls u, \varphi \rangle = \int_{ \Omega}( \Ls u) \varphi
\end{equation}
for all $\varphi \in \mathcal C_c^\infty (\Omega)$, i.e. $\widetilde \Ls u$ has a Riesz representation as a pointwise function (see, e.g., \cite{stein2016singular,stein2016fourier}). In this sense, we can ensure that any very weak solution of \eqref{eq:FDE Laplace} satisfies
\begin{equation}
	\widetilde \Ls u = f \qquad \textrm{ in } \mathcal D' (\Omega).
\end{equation}
In fact, this distributional extension is precisely the one that comes naturally from the very weak solutions used in this paper.


\section{Auxiliary result. Boundary behaviour of $\varphi_{ \delta }$} \label{sec:phi delta estimates}

\begin{proof}[Proof of \Cref{lem:phi delta estimates}] \label{proof lem:phi delta estimates}
	We know that
	\begin{equation} \label{eq:varphi delta kernel}
		\varphi_{ \delta } (x) = \int_{ \Omega } \frac{\mathbb G_s (x,y)}{\delta^s(y)} dy.
	\end{equation}
	Due to \eqref{eq:Green function estimates}
	\begin{equation}
		\mathbb G_s (x,y) \le \frac{c}{|x-y|^{n-2s}} \min \left( \frac{\delta^s(x)}{|x-y|^s},1 \right).
	\end{equation}
	To estimate the behaviour of $\varphi_{ \delta }$ near the boundary we take a point $x$ near $\partial \Omega$ and consider the integral in a small ball $B$ with center $x$ and radius $\frac{\delta(x)}{2}$. We split $\varphi_{ \delta } (x) = I_1 + I_2$ by splitting the integral \eqref{eq:varphi delta kernel} into integrals in $B$ and $\Omega\setminus B$. We have
	\begin{equation}
		 I_1 \defeq \int_{ B } \frac{\mathbb G_s (x,y)}{\delta^s(y)} dy \le \int_{ B } \frac{c}{|x-y|^{n-2s}\delta^s(y)} dy.
	\end{equation}
	On the other hand, in $B$, $\delta(y) \ge \delta - \frac{\delta(x)}{2} \ge c \delta(x)$ and hence
	\begin{equation}
		I_1 \le \frac{c}{\delta^s(x)} \int_{ B }  \frac{1}{|x-y|^{n-2s}} dy.
	\end{equation}
	Integrating in spherical coordinates
	\begin{equation}
		I_1 \le \frac{c}{\delta^s(x)} \int_{ 0 }^{\frac{\delta(x)}{2}}  \frac{1}{r^{n-2s}} r^{n-1}dr =  \frac{c}{\delta^s(x)} r^{2s} \Big|_{ 0 }^{\frac{\delta(x)}{2}} \le c \delta^s(x).
	\end{equation}
	On the other hand we have that
	\begin{align}
		I_2 &\defeq \int_{ \Omega \setminus B } \frac{\mathbb G_s(x,y)}{\delta(y)^s} dy \\
		 &\le c \int_{|x-y|\ge \frac{\delta(x)}{2}} \frac{1}{|x-y|^{n-2s} \delta(y)^s} \frac{\delta(x)^s}{|x-y|^{s}} dy \\
		 &= c \delta^s (x) \int_{|x-y|\ge \frac{\delta(x)}{2}} \frac{1}{|x-y|^{n-s} \delta(y)^s}  dy \\
		 &\le c \delta^s (x) \int_{|x-y|\ge \frac{\delta(x)}{2}} \frac{1}{|x-y|^{n}}  dy.
	\end{align}
	Let $R = \max_{y \in \Omega} |x-y|$. We can integrate radially to compute
	\begin{align}
		 \int_{|x-y|\ge \frac{\delta(x)}{2}} \frac{1}{|x-y|^{n}}  dy &\le c \int_{\frac{\delta(x)}{2}}^R \frac{1}{r^n} r^{n-1} dr \\
		 &=  c \left(  \log R - \log \frac{\delta(x)}{2}  \right)\\
		 &\le c (1 + |\log \delta(x)|).
	\end{align}
	Thus
	\begin{equation}
		I_2 \le c \delta^s(x) ( 1 + |\log \delta(x)| ).
	\end{equation}
	This concludes the upper bound for $\varphi_{ \delta }$.
	
	On the other hand $I_1, I_2 \ge 0$. For the lower bound we look only at $I_1$. Due to \eqref{eq:Green function estimates} we also have that
	\begin{equation}
		\mathbb G_s (x,y) \ge \frac{c}{|x-y|^{n-2s}} \min \left( \frac{\delta(x)^s}{|x-y|^s},1 \right)\min \left( \frac{\delta(y)^s}{|x-y|^s},1 \right).
	\end{equation}
	Here we have to be a bit more careful with the minimum. In $B$, $\frac{\delta(x)^s}{|x-y|^s} \ge 2^s \ge 1$. Also, $\delta(y) \ge \frac{\delta (x)}{2}$ and so $\frac{\delta(y)^s}{|x-y|^s} \ge 1$ in $B$. Hence
	\begin{equation}
		\mathbb G_s (x,y) \ge \frac{c}{|x-y|^{n-2s}}, \qquad \textrm{if }|x-y| \le \frac{\delta(x)}{2}.
	\end{equation}
	Therefore
	\begin{equation}
		I_1 \ge \int_{ B } \frac{c}{|x-y|^{n-2s}\delta^s(y)} dy \ge \frac{c}{\delta^s(x)} \int_{ 0 }^{\frac{\delta(x)}{2}}  \frac{1}{r^{n-2s}} r^{n-1}dr =  \frac{c}{\delta^s(x)} r^{2s} \Big|_{ 0 }^{\frac{\delta(x)}{2}} \ge c \delta^s(x).
	\end{equation}
	This concludes the proof.
\end{proof}


\section{An alternative proof of Kato's inequality for the fractional Laplacian with weight}
\label{section:Kato proof}

	\begin{proposition}[Kato's inequality]   \label{Kato}
Let $u \in \mathcal C^{2s} (\mathbb R^n)$, then, for every $x \in \mathbb R^n$
		\begin{align}
		(- \Delta)^s u_+ &\le \sign_+u \, (-\Delta)^s u \\
		(- \Delta)^s |u| &\le \sign u \, (-\Delta)^s u .
		\end{align}
		Moreover, if $u\in L^1(\Omega)$, $f \delta^s \in L^1 (\Omega)$ and assuming that
		\begin{equation} \label{eq:Kato hypothesis}
		\int_{ \Omega } u (-\Delta)^s \varphi = \int_{ \Omega } f \varphi \qquad \forall 0 \le \varphi \in \Xs \cap C_c (\Omega),
		\end{equation}
		then, there exist $\xi_+ \in \widetilde \sign_+ (u)$ and $\xi \in \widetilde \sign (u)$ such that
		\begin{align}
		\int_{ \Omega } u_+ (-\Delta)^s \varphi &\le \int_{ \Omega } \xi_+ f  \, \varphi \\
		\int_{ \Omega } |u| (-\Delta)^s \varphi &\le \int_{ \Omega } \xi f  \, \varphi,
		\end{align}
		for all $0 \le \varphi \in \Xs \cap C_c  (\Omega)$, where $\widetilde \sign_+$ and $\widetilde \sign$ are the maximal monotone graphs given by
		\begin{equation}
		\widetilde \sign_+ (s) = \begin{dcases}
		0 & s < 0 , \\
		[0, 1] & s= 0, \\
		1 & s\ge 0.
		\end{dcases} \qquad
		\widetilde \sign (s) = \begin{dcases}
		-1 & s < 0 , \\
		[-1, 1] & s= 0, \\
		1 & s\ge 0.
		\end{dcases}
		\end{equation}
	\end{proposition}
	\begin{proof}
	First assume $u \in \mathcal C^{2s} (\mathbb R^n)$. Let $s_+(x) = \sign_+ u(x)$. We have that
	\begin{align}
	u_+(y) &\ge s_+(x) u(y) , \\
	u_+(x) &= s_+(x) u(x), \\
	\frac{u_+(x)-u_+(y)}{|x-y|^{n+2s}} &\le s_+(x) \frac{u(x)-u(y)}{|x-y|^{n+2s}} \\
	\int_ \Omega \frac{u_+(x)-u_+(y)}{|x-y|^{n+2s}} dy  &\le s_+(x) \int_ \omega \frac{u(x)-u(y)}{|x-y|^{n+2s}}dy\\
	(-\Delta)^s u_+ (x) &\le s_+(x) (-\Delta)^s u(x).
	\end{align}
	Applying this result to $-u$:
	\begin{align}
		(-\Delta)^s (-u)_+ (x) &\le \sign_+(-u) (-\Delta)^s (-u)\\
		(-\Delta)^s u_- (x) &\le \sign_-(u) (-\Delta)^s u.
	\end{align}
	Therefore,
	\begin{equation}
		(-\Delta)^s |u| \le \sign (u)\, (-\Delta)^s u.
	\end{equation}
	If $0 \le \varphi \in \Xs \cap C_c(\Omega)$ we have
	\begin{align}
	\int_ \Omega u_+(x)  (-\Delta)^s  \varphi (x) dx &= \int_ \Omega (-\Delta)^s u_+ (x) \varphi (x) dx \le  \int _\Omega s_+(x) f(x)  \varphi (x)dx.
	\label{eq:Kato proof 5}
	\end{align}
	
	Assume now that $u \in L^1(\Omega)$, $u = 0$ in $\Omega^c$ and \eqref{eq:Kato hypothesis} holds.
	Let $f_k = T_k (f)$ we have that $f_k \delta^s  \to f \delta^s $ in $L^1 (\Omega)$. Let $u_k$ be the unique solutions of
	\begin{equation}
	\begin{dcases}
	(-\Delta)^s u_k = f_k & \Omega, \\
	u_k = 0 & \Omega^c .
	\end{dcases}
	\end{equation}
	Then, by the results in \cite{chen+veron2014}, we know that $u_k \to u$ in $L^1(\Omega)$, hence $(u_k)_+ \to u_+$ in $L^1 (\Omega)$. On the other hand, by the previous part of the proof
	\begin{equation}
	\int_ \Omega(u_k )_+(x)  (-\Delta)^s  \varphi (x) dx  \le \int _\Omega \sign_+(u_k(x)) f_k (x)  \varphi (x) dx, \qquad  \forall 0 \le\varphi \in\Xs \cap C_c (\Omega).
	\end{equation}
	Let $0 \le \underline \gamma_ \ee (s) \le \sign_+ (s) \le \overline \gamma_ \ee (s) \le 1$ be smooth functions
	\begin{equation}
	\overline \gamma_ \ee (s) = \begin{dcases}
	0 & s < -\ee , \\
	1 & s > 0.
	\end{dcases}, \qquad
	\underline \gamma_ \ee (s) = \begin{dcases}
	0 & s < 0 , \\
	1 & s > \ee.
	\end{dcases}
	\end{equation}
	Since $f(x) > 0$ if and only $f_k (x) > 0$, we have that
	\begin{equation}
		\int_ \Omega(u_k )_+(x)  (-\Delta)^s  \varphi (x) dx  \le \int _{ \{f\ge 0\} } \overline \gamma_\ee (u_k(x)) f_k (x)  \varphi (x) dx + \int _{ \{f<0\} } \underline \gamma_\ee (u_k(x)) f_k (x)  \varphi (x) dx,
	\end{equation}
	for all $ 0 \le\varphi \in\Xs \cap C_c (\Omega)$. As $k \to \infty$ we have that
	\begin{equation}
		\int_ \Omega u_+(x)  (-\Delta)^s  \varphi (x) dx  \le \int _{ \{f\ge 0\} } \overline \gamma_\ee (u(x)) f (x)  \varphi (x) dx + \int _{ \{f<0\} } \underline \gamma_\ee (u(x)) f(x)  \varphi (x) dx.
	\end{equation}
	Up to a subsequence, there exists $\xi_+ \in L^\infty (\Omega)$ such that
	\begin{equation}
		\overline \gamma_\ee (u(x)) \chi_{\{f\ge 0 \}} + \underline \gamma_\ee (u(x)) \chi_{\{f<0\}} \to \xi_+(x) \qquad \textrm{ in } L^\infty\textrm{-weak-}\star.
	\end{equation}
	By the pointwise limits $\xi_+(x) = \sign_+ (u(x))$ when $u(x) \ne 0$ and $0 \le \xi_+ \le 1$. Thus $\xi_+ (x) \in \widetilde \sign (u(x))$.\\

	Hence
	\begin{equation}
	\int_ \Omega u_+(x)  (-\Delta)^s  \varphi (x) dx  \le \int _\Omega \xi_+ (x) f (x)  \varphi (x) dx, \qquad  \forall 0 \le\varphi \in\Xs \cap C_c (\Omega).
	\end{equation}
	As for the pointwise estimate, we can proceed analogously for $u_-$ (where $u = u_+ - u_-$) to deduce that
	\begin{equation}
	\int_ \Omega u_-(x)  (-\Delta)^s  \varphi (x) dx  \le \int _\Omega \xi_- f(x)  \varphi (x) dx, \qquad  \forall 0 \le\varphi \in\Xs \cap C_c (\Omega),
	\end{equation}
	with $\xi_-(x) \in \widetilde \sign_- (u(x))$. We then have
	\begin{equation}
	\int_ \Omega |u(x)|  (-\Delta)^s  \varphi (x) dx  \le \int _\Omega \xi (x) f(x)  \varphi (x) dx, \qquad  \forall 0 \le\varphi \in\Xs \cap C_c (\Omega),
	\end{equation}
	where $\xi (x) = \xi_+ (x) + \xi_-(x) \in \widetilde \sign (u(x))$. This concludes the proof.
	\end{proof}


\section{The weighted approach for related parabolic problems}

 The combination of our well-posedness results and a priori estimates allow us to immediately solve a number of related evolution problems, according to a general procedure of the evolution theory.

\noindent  {\bf 1.} The initial-value parabolic problem
\begin{equation}
\begin{dcases}
	\partial_t u + (-\Delta )^{s}  u + V(x)u=f(x,t) & \Omega \times (0,T)\\
	u = 0 & \Omega^c \times [0,T) , \\
	u = u_0 & \Omega \times \{0\},
	\end{dcases}
\end{equation}
can be solved for every $u_0\in L^1(\Omega; \phi)$, $f\in L^1(0,T; L^1(\Omega; \phi)$
under the conditions $0<s<1$, $V\in L^1_{loc}(\Omega), $ $V\le 0$  and $\phi$ is a positive  weight  in $X^s$ such that $(-\Delta)^s \phi\ge 0$.

Using \Cref{cor.accret} and the Crandall-Liggett generation theorem \cite{crandall1971generation} a contraction semigroup in all such spaces is generated and it satisfies the Maximum Principle.

Note that the fractional heat equation (case $V=0)$ has been studied in the whole space $\mathbb R^n$ in an optimal class of weighted integrable data in \cite{MR3614666}. The optimal weighted space in which solutions of the Cauchy problem for $\partial_t u+(-\Delta)^s u=0$ are well-posed is
$$
\int \frac{|u_0(x)|}{(1+|x|^2)^{(n+2s)/2}},dx<\infty.
$$
The reader will notice that  the weight decays at infinity in a precise way, to be compared with the behaviour $\delta^s$ of the bounded case.

The considerations made in \cite{diaz2019ambiguous} for the associated complex relativistic Schr\"odinger problem with potentials $V = \delta^{-2s}$ can be extended to the case of supersingular potentials $V \ge c \delta^{-2s}$, thanks to the results of Section 4 of this paper.

\noindent  {\bf 2. Fractional-PME}
The same project  can be applied to the {\sl fractional porous medium equation}
$$
\partial_t u+(-\Delta)^s u^m=f, \quad
$$
with $m>0$, $m\ne 1$, that has been studied in many works, mainly when $f=0$.
Thus, the non-weighted theory is done in \cite{dPQRV-MR2737788, dPQRV-MR2954615, Vazquez2014, Bonforte2015,  bonforte1610sharp}.
The basic result of generation of a semigroup in $L^1$ goes back to \cite{CrandallMR647071} and was used in \cite{BonfMR3427986}. The weighted theory is to be done.

Much work remains to be done on these issues.

\section{Comments, extensions, and open problems}\label{sec.comment}

Here are  some issues motivated by the previous presentation.

\subsection{More general potentials}
In this paper we have considered  nonnegative potentials $V\in L^1_{loc}(\Omega)$. This allows for extensions in two directions: considering signed potentials, and considering locally bounded measures as potentials. Both are present in the literature, but both lead to problems that we do not want to consider here.

\subsection{Other fractionary and nonlocal operators}

When working in bounded domains, there are several different choices of $(-\Delta)^s$ present in the literature (see, e.g., \cite{Bonforte2015, BFV2018, MR3588125, MusinaMR3246044, servadei2014spectrum}. The main choices apart from the restricted Laplacian treated here are the spectral Laplacian and censored Laplacian, ... Many of our results can be extended to them and this is contents of future work. Note that regularity for the equation $Lu=f$ in the case of the spectral Laplacian was studied in  {\cite{Caffarelli+Stinga2016}}.

Another issue is the Klein-Gordon fractional operator considered in Quantum Mechanics $\sqrt{(-\Delta ) + m^2}\,u$, and  mentioned in the Introduction. The theory for this operator is quite similar to what we have exhibited above for $(-\Delta)^{1/2}$, see \cite{diaz2019ambiguous}.  

The theory of this paper can be developed for more related  integro-differential operators that are being investigated like the integro-differential operators with irregular or rough kernels, as in \cite{karch2010nonlinear}. \normalcolor

The behaviour of the typical solutions of these operators near the boundary makes a difference. Thus, solutions of equations involving the spectral Laplacian they satisfy the linear behaviour of the classical Hopf principle, i.e., linear growth near the boundary.

	\subsection{Associated eigenvalue problem}
A main question for the Schr\"odinger equation is the eigenvalue problem, which comes from separation of variables. The eigenvalue theory works well in the sense of weak solutions in $L^2 (\Omega)$. For the classical Schr\"odinger problem with $s=1$, it is known that the eigenvalues of $L^1 (\Omega)$ and $L^2 (\Omega)$ are not the same. See \cite{Cabre+Martel1998}. It would be interesting to know if such difference remains being true for $s<1$.

\subsection{Open Problem on further integrability of the solutions}
		For problem \eqref{eq:FDE Laplace}, via the estimates on the Green kernel \eqref{eq:Green operator}, the natural space for integrability will be of the form $W^{s',p} (\Omega, \delta^s)$ for $s' < s$ and $p$ small.  These estimates can then be extended to problem \eqref{eq:FDE} using maybe the methods of \cite{Diaz+Rakotoson:2009, Diaz+Rakotoson:2010}.

\section*{Acknowledgments} The authors have been partially funded by the Spanish Ministry of  Economy, Industry and Competitiveness
 under projects MTM2014-57113-P and MTM2014-52240-P. J.\ I. D\'iaz and D. G\'omez-Castro are members of	the Research Group MOMAT (Ref. 910480) of the UCM. J.\ L. V\'azquez would like to thank IMI (Instituto de Matem\'atica Interdisciplinar) for their kind invitation to visit the Universidad Complutense in the academic year 2017--2018.  We want to  thank X. Ros-Oton and Y. Sire for interesting observations on the contents of the paper.\normalcolor

\

\

\noindent {\sc Keywords.}  Nonlocal elliptic equations, bounded domains, Schr\"odinger operators, super-singular potentials, very weak solutions, weighted spaces. \normalcolor

\noindent{\sc Mathematics Subject Classification}. 35J10, 35D30,  35J67,
35J75.

\end{document}